\documentclass[a4paper,reqno,10pt]{amsart}
\usepackage[utf8]{inputenc}
\usepackage{amsfonts}
\usepackage{amsmath}
\usepackage{amsthm}
\usepackage{amssymb}
\usepackage{graphicx} 
\usepackage[dvipsnames]{xcolor}
\usepackage{mlmodern}
\usepackage[linktocpage,bookmarksopen=true]{hyperref}

\newtheoremstyle{teoremas}
{10pt}
{10pt}
{\itshape}
{}
{\bfseries}
{}
{.5em}
{}

\theoremstyle{teoremas}
\newtheorem{theorem}{Theorem}[section]

\newtheorem{corollary}[theorem]{Corollary}

\newtheorem{lemma}[theorem]{Lemma}
\newtheorem{proposition}[theorem]{Proposition}

\numberwithin{equation}{section}
\newtheorem{algorithm}{Algorithm}[section]

\newtheoremstyle{definition}
{10pt}
{10pt}
{}
{}
{\bfseries}
{}
{.5em}
{}

\theoremstyle{definition}
\newtheorem{definition}[theorem]{Definition}

\newtheorem{conjecture}[theorem]{Conjecture}

\newtheorem{example}[theorem]{Example}

\newtheorem{remark}[theorem]{Remark}

\newcommand\bbinom[2]%
{\mathchoice{\left(\kern-0.48em{\binom{#1}{#2}}\kern-0.48em\right)}
            {\bigl(\kern-0.3em{\binom{#1}{#2}}\kern-0.3em\bigr)}
            {\bigl(\kern-0.3em{\binom{#1}{#2}}\kern-0.3em\bigr)}
            {\bigl(\kern-0.3em{\binom{#1}{#2}}\kern-0.3em\bigr)}}

\usepackage{tikz}
\usetikzlibrary{automata, arrows}
\usepackage[scr=boondoxo]{mathalfa}
\usepackage{tcolorbox}
\tcbuselibrary{theorems}
\newtcbtheorem
  []
  {guide}
  {Guide to the proof}
  {%
    colback=white!5,
    colframe=BrickRed!45!white,
    fonttitle=\bfseries,
  }
  {guide}
\newtcbtheorem
  []
  {protocol}
  {Protocol}
  {%
    colback=white!5,
    colframe=NavyBlue!45!white,
    fonttitle=\bfseries,
  }
  {prot}

\usepackage[shortlabels]{enumitem}
\usepackage{multicol}
\usepackage{caption}

\hypersetup{
    colorlinks = true,
    linkbordercolor = {white},
    linkcolor = {NavyBlue},
    anchorcolor = {black},
    citecolor = {NavyBlue},
    filecolor = {cyan},
    menucolor = {NavyBlue},
    runcolor = {cyan},
    urlcolor = {NavyBlue}
}

\usepackage[colorinlistoftodos]{todonotes}

\usepackage{newverbs}
\newverbcommand{\code}{\color{Green}}{}

\usepackage[margin=1.25in]{geometry} 
\definecolor{color_petroleo}{HTML}{9400e3}
\definecolor{color_gris}{HTML}{9dA1A0}
\definecolor{color_car}{HTML}{FF006A}
\definecolor{color_car_luky}{HTML}{00e30c}

\newcommand{\carroTikZ}[6]{
    
    \pgfmathsetmacro{\ys}{-1 * #4}
    \pgfmathsetmacro{\xs}{#4}

    \begin{scope}[shift={(#2,#3)}, rotate=#5, yscale = \ys, xscale = \xs, scale=0.02 ]
        \begin{scope}[shift={(-198,-156.5)}]
            \draw [draw opacity=0][fill=#1  ,fill opacity=1 ]   (127.27,155.47) .. controls (135.27,130.47) and (150.93,109.47) .. (189.27,108.13) .. controls (227.6,106.8) and (231.93,131.8) .. (239.93,136.47) .. controls (247.93,141.13) and (272.6,137.8) .. (273.27,168.47) .. controls (282.27,171.13) and (277.93,183.8) .. (267.6,181.13) .. controls (218.27,180.8) and (194.93,180.8) .. (151.27,181.47) .. controls (122.6,180.8) and (107.93,182.13) .. (125.27,163.47) ;
            \draw  [draw opacity=0][fill={rgb, 255:red, 0; green, 0; blue, 0 }  ,fill opacity=1 ] (136.93,175.47) .. controls (136.93,165.53) and (145.29,157.47) .. (155.6,157.47) .. controls (165.91,157.47) and (174.27,165.53) .. (174.27,175.47) .. controls (174.27,185.41) and (165.91,193.47) .. (155.6,193.47) .. controls (145.29,193.47) and (136.93,185.41) .. (136.93,175.47) -- cycle ;
            \draw  [draw opacity=0][fill={rgb, 255:red, 218; green, 215; blue, 215 }  ,fill opacity=1 ] (145.93,175.47) .. controls (145.93,170.32) and (150.26,166.15) .. (155.6,166.15) .. controls (160.94,166.15) and (165.27,170.32) .. (165.27,175.47) .. controls (165.27,180.61) and (160.94,184.79) .. (155.6,184.79) .. controls (150.26,184.79) and (145.93,180.61) .. (145.93,175.47) -- cycle ;
            \draw  [draw opacity=0][fill={rgb, 255:red, 0; green, 0; blue, 0 }  ,fill opacity=1 ] (223.6,175.13) .. controls (223.6,165.19) and (231.96,157.13) .. (242.27,157.13) .. controls (252.58,157.13) and (260.93,165.19) .. (260.93,175.13) .. controls (260.93,185.07) and (252.58,193.13) .. (242.27,193.13) .. controls (231.96,193.13) and (223.6,185.07) .. (223.6,175.13) -- cycle ;
            \draw  [draw opacity=0][fill={rgb, 255:red, 218; green, 215; blue, 215 }  ,fill opacity=1 ] (232.6,175.13) .. controls (232.6,169.99) and (236.93,165.81) .. (242.27,165.81) .. controls (247.61,165.81) and (251.93,169.99) .. (251.93,175.13) .. controls (251.93,180.28) and (247.61,184.45) .. (242.27,184.45) .. controls (236.93,184.45) and (232.6,180.28) .. (232.6,175.13) -- cycle ;
            \draw   (177.6,156.8) .. controls (177.6,145.57) and (186.7,136.47) .. (197.93,136.47) .. controls (209.16,136.47) and (218.27,145.57) .. (218.27,156.8) .. controls (218.27,168.03) and (209.16,177.13) .. (197.93,177.13) .. controls (186.7,177.13) and (177.6,168.03) .. (177.6,156.8) -- cycle ;
            \draw  [fill={rgb, 255:red, 230; green, 230; blue, 230 }  ,fill opacity=1 ] (151,132.72) .. controls (151,132.65) and (151,132.59) .. (151,132.53) .. controls (151,120.75) and (168.31,111.2) .. (189.67,111.2) .. controls (211.02,111.2) and (228.33,120.75) .. (228.33,132.53) .. controls (228.33,132.75) and (228.33,132.97) .. (228.32,133.19) -- cycle ;
            \draw  [draw opacity=0][fill=#1  ,fill opacity=1 ] (193.33,109.53) -- (197.43,109.53) -- (197.43,135.2) -- (193.33,135.2) -- cycle ;
            \draw    (193,111) -- (193,133) ;
            \draw    (198,111) -- (198,133) ;
            \node[draw=black,circle, fill = {rgb, 255:red, 255; green, 255; blue, 255 },fill opacity=0.8,minimum size=2.3em,inner sep=1pt, scale = #4, rotate = #5]  at (198,156.5) {};
            \node[scale = #4, rotate = #5] at (198,156.5) {\huge #6};
        \end{scope}
    \end{scope}
}

\newcommand{\spotCircular}[7]{
  \begin{scope}[shift={(#1,#2)}]
    \path[fill=black, draw=none, rounded corners=3pt]
      (#5:#3) arc (#5:#6:#3)
      -- (#6:{#3 + #4}) arc (#6:#5:{#3 + #4})
      -- cycle;
    \pgfmathsetmacro{\angleMid}{(#5 + #6)/2}
    \pgfmathsetmacro{\radiusMid}{#3 + #4/2}
    \node[text = white, rotate = \angleMid - 90] at (\angleMid:\radiusMid) {#7};

  \end{scope}
}

\newcommand{\preferenciaCircular}[7]{%
  \begin{scope}[shift={(#1,#2)}]
    \path[fill=color_petroleo, draw=none, rounded corners=3pt]
      (#5:#3) arc (#5:#6:#3)
      -- (#6:{#3 + #4}) arc (#6:#5:{#3 + #4})
      -- cycle;
    \pgfmathsetmacro{\angleMid}{(#5 + #6)/2}
    \pgfmathsetmacro{\radiusMid}{#3 + #4/2}
    \node[text = white, rotate = \angleMid - 90] at (\angleMid:\radiusMid) {#7};

  \end{scope}
}

\DeclareMathOperator{\vol}{vol}

\newcommand{\ehr}{\operatorname{ehr}}
\renewcommand{\vol}{\operatorname{vol}}

\newcommand{\Z}{\mathbb{Z}}

\newcommand{\Pf}{\operatorname{PF}}
\newcommand{\zz}{\mathbf{z}}
\newcommand{\yy}{\mathbf{y}}
\renewcommand{\ss}{{\bf s}}
\newcommand{\aaa}{{\bf a}}
\newcommand{\yPerm}{\mathcal{P}}
\newcommand{\Draco}{\operatorname{Drac}}

\newcommand{\ExtPerm}[2]{\mathfrak{S}_{#1,#2}}
\newcommand{\mulPerm}[1]{\mathfrak{S}_{#1}}

\newcommand{\alg}{\operatorname{alg}}

\newcommand{\blucky}{\operatorname{lucky}^{\circlearrowright}}

\newcommand{\Lucky}{\operatorname{Lucky}}
\newcommand{\lucky}{\operatorname{lucky}}
\newcommand{\out}{\operatorname{out}}

\renewcommand{\emptyset}{\varnothing}

\usepackage{graphicx}

\title[Luck and magic for Pitman--Stanley polytopes and parking functions]{Luck and magic for Pitman--Stanley polytopes\\ and parking functions}

\address{(N. Avila \& A. H. Morales)
Universit\'e du Qu\'ebec \`a Montr\'eal, Montr\'eal, Canada}
\email{morales\_borrero.alejandro@uqam.ca}
\email{avila\_ramirez.nicolas@courrier.uqam.ca}

\author{Nicolas Avila, Luis Ferroni, Alejandro H. Morales}

\address{(L. Ferroni)
    Universit\`a di Pisa, Pisa, Italy
}
\email{luis.ferroni@unipi.it}


\thanks{LF is a member of the INDAM research group GNSAGA. NA and AHM are partially supported by the NSERC Discovery grant RGPIN-2024-06246 and the FRQNT team grant 10.69777/341288. }

\begin{document}
\allowdisplaybreaks

\begin{abstract}
    Motivated by the combinatorics of parking functions and their several generalizations, we study the Ehrhart theory of Pitman--Stanley polytopes. We prove a strong positivity phenomenon called \emph{magic positivity} for the Ehrhart polynomials of these polytopes, which in turn implies that their $h^*$-polynomials are real-rooted (and thus log-concave and unimodal). Our result is achieved by interpreting the coefficients of these Ehrhart polynomials in the \emph{magic basis} in terms of the number of \emph{lucky cars} in a modified parking protocol. Furthermore, we address the magic positivity problem for $\yy$-generalized permutohedra and also discuss a \emph{magic} combinatorial interpretation for them, under the assumption that the input parameters are sufficiently large.
\end{abstract}


\maketitle

\section{Introduction}

\subsection{Overview}

In an influential paper, Pitman and Stanley \cite{Pitman_Stanley_1999} introduced a class of lattice polytopes, now commonly called \emph{Pitman--Stanley polytopes}. The precise definition of a Pitman--Stanley polytope is as follows. Consider a vector $\mathbf{y}\in \mathbb{Z}_{\geq  0}^n$, and define
    \[ \Pi_n(\mathbf{y}) := \left\{ \mathbf{x} \in \mathbb{R}_{\geq 0}^n : \enspace\sum_{i=1}^k x_i \leq \sum_{i=1}^k y_i \text{ for each  $k=1,\ldots,n$}\right\}.\]
(We customarily use the non-bold names $x_i$ and $y_i$ to denote the $i$-th coordinate of the boldface vectors $\mathbf{x}$ and $\mathbf{y}$ respectively.) These polytopes are especially relevant due to their appearance in diverse contexts within probability, statistics, algebraic combinatorics, and discrete geometry. Their face structure, volume computation, and lattice point enumerations were addressed in the original work by Pitman and Stanley, but many aspects of the theory are still attracting mathematicians to this topic (see \cite{KonvPak,KonvPak2,Baldoni_Vergne_2008,genPFPoly,gen_PS1}).

In the present article we shall be concerned with the enumeration of lattice points in (dilations of) Pitman--Stanley polytopes.  A celebrated theorem of Ehrhart \cite{ehrhart} establishes that the number of lattice points in the $m$-th dilation of a lattice polytope $\mathcal{P}\subseteq \mathbb{R}^n$ is interpolated by a polynomial of degree $d=\dim \mathcal{P}$. In other words, there exists a polynomial $\ehr_{\mathcal{P}}\in \mathbb{R}[t]$ such that
    \[ \ehr_{\mathcal{P}}(t) := |t\mathcal{P} \cap \mathbb{Z}^n|, \quad \text{ for every positive integer $t$.}\]
There is considerable interest in studying the coefficients of the Ehrhart polynomial of a polytope. If we write 
\begin{equation} \label{eq:standard-basis}
\ehr_{\mathcal{P}}(t) = a_d\, t^d + a_{d-1}\, t^{d-1} + \cdots + a_1\, t + a_0,
\end{equation}
then the following facts are known:
    \begin{itemize}
        \item The leading coefficient $a_d$ equals the relative volume of $\mathcal{P}$.
        \item The second leading coefficient $a_{d-1}$ equals half the relative volume of the boundary $\partial \mathcal{P}$.
        \item The constant term $a_0$ equals the Euler characteristic of the Euclidean space, i.e., $1$.
    \end{itemize}

A consequence of the interpretation for $a_d$ is that the Ehrhart polynomial of $\mathcal{P}$ provides a discrete analog for the notion of volumes. The computation of volumes of polytopes constitutes by itself a recurring and crucial problem in many areas of mathematics. 

Due to the geometric interpretations for the three coefficients above, one may be inclined to believe that something analogous happens for the remaining coefficients. In reality, the coefficients $a_1, a_2,\ldots, a_{d-2}$ are in general very complicated. There exist general closed formulas by McMullen \cite{mcmullen}, but they depend on certain choices and are not easy to state. Not even the sign pattern for these coefficients is fully understood. As it turns out, they can be negative (and all simultaneously at once, see \cite{hibi-higashitani-yoshida-tsuchiya}). When all the coefficients $a_1,\ldots,a_{d-2}$ are non-negative, we say that $\mathcal{P}$ is Ehrhart positive. We recommend \cite{liu, ferroni-higashitani} for detailed surveys addressing Ehrhart positivity.

One of the results of Pitman and Stanley is a closed formula that shows that for any vector $\mathbf{y}$ the polytope $\Pi_n(\mathbf{y})$, which is $n$-dimensional if $y_1>0$, is Ehrhart positive. Specifically, their formula reads as follows:
    \[ \ehr_{\Pi_n({\bf y})}(t) = 
 \sum_{{\bf s} \in \mathcal{I}_n} \binom{t\cdot y_1+s_1}{s_1} \binom{t\cdot y_2+s_2-1}{s_2}\cdots \binom{t\cdot y_n+s_n-1}{s_n},
    \]
where the sum runs over a set of vectors of non-negative integers (see Corollary~\ref{coro:ehr-PS} below for the details). Ehrhart positivity follows immediately from having a formula of this kind: the reason being that each summand is itself a polynomial with non-negative coefficients in $t$. It is nonetheless difficult to extract nice formulas for each individual coefficient of this polynomial, and no \emph{nice} combinatorial interpretation is known for them (after multiplying by $(n-1)!$ which therefore turns each coefficient into a non-negative integer).

A considerable strengthening of Ehrhart positivity is the positivity of the Ehrhart polynomial when written in an alternative basis. If instead of using the monomial basis $1, t, t^2,\ldots, t^d$, we use the basis $\{t^i(1+t)^{d-i}\}_{i=0}^d$, we can write
    \begin{equation} \label{eq:magic-basis}
    \ehr_{\mathcal{P}}(t) = \sum_{i=0}^d c_i\, t^i(1+t)^{d-i} = c_0\, t^0(1+t)^d + c_1 t^1 (1+t)^{d-1} + \cdots +c_d t^d(1+t)^{0}.
    \end{equation}
for some real numbers $c_0,\ldots,c_d$. The nonnegativity of all the $c_i$ in equation~\eqref{eq:magic-basis} implies the nonnegativity of the $a_i$ in equation~\eqref{eq:standard-basis}. In this case, the coefficients $c_i$ refine the volume of $\mathcal{P}$. Following Ferroni and Higashitani \cite{ferroni-higashitani}, we say that $\ehr_{\mathcal{P}}(t)$ is \emph{magic positive} if $c_i\geq 0$ for $i=0,\ldots,d$. The main contribution in this article is proving that Pitman--Stanley polytopes have Ehrhart polynomials that are magic positive. This result was announced in \cite{ferroni-higashitani}, and our extended abstract \cite{avila-ferroni-morales-abstract} containing some details of the proof has circulated among colleagues in the past months.

The specific strategy we employ to prove this result yields in fact something far more interesting than the mere nonnegativity of the coefficients in equation~\eqref{eq:magic-basis}. As it turns out, there is a nice \emph{combinatorial interpretation} for these coefficients. In order to state our interpretation, we need to rely on one of the many frameworks in which Pitman--Stanley polytopes arise: that of $\mathbf{y}$-parking functions (see Definition~\ref{def:y-parking-functions} below for a precise statement), a generalization of classical parking functions, whose number gives the normalized volume of $\Pi_n(\yy)$ (see Theorem~\ref{thm:vol PS in terms of PF}).

\begin{theorem}\label{thm:main-intro} 
    Let $\Pi_n(\mathbf{y})$ be a Pitman--Stanley polytope for $\mathbf{y} \in \mathbb{Z}^n_{>0}$. Let 
        \[ \ehr_{\Pi_n(\mathbf{y})}(t) = \sum_{i=0}^n c_i\, t^i(1+t)^{n-i}.\]
    Then, for each $0\leq i \leq n$, the number $n!\cdot c_i$ enumerates the $\yy$-parking functions having exactly $i$ lucky cars without the first available space. In particular, each $c_i$ is a nonnegative number, and therefore magic positivity holds.
\end{theorem}

The case ${\bf y}={\bf 1}$ recovers a result by Gessel and Seo \cite[Thm. 10.1]{GesselSeo}. For an explanation of the terms ``lucky cars'' and ``first available space'' we refer the reader to Section~\ref{sec:magic}; see also \cite{StanleyYin,harris2024parkingfunctionsfixedset, kenyonparkingfunctions, ferreri2025enumeratingvectorparkingfunctions} where the authors use the term ``lucky cars'' in the context of parking functions and their generalizations, including $\yy$-parking functions. In the case that ${\yy} \in \mathbb{Z}^n_{>0}$, this appears to be related to a different lucky statistic on $\yy$-parking functions independently studied by Ferreri--Harris--Martinez--Swartz in \cite{ferreri2025enumeratingvectorparkingfunctions} (see Section~\ref{sec: other parking protocols y-parking} and Conjecture~\ref{conj: equidistribution both protocols}).

The class of magic positive Ehrhart polynomials is attracting considerable attention nowadays (see for instance \cite{ferroni-higashitani,branden-ferroni-jochemko,konoike,konoike2,DHS,liu-xiao}, being the case of Pitman--Stanley polytopes one of the most anticipated ones. Furthermore, by relying on some ideas used in the proof of Theorem~\ref{thm:main-intro} (which appeared in condensed form in our extended abstract \cite{avila-ferroni-morales-abstract}), Athanasiadis, Xiao, and Yan \cite{athanasiadis-xiao-yan} proved very recently that a certain class of polytopes called \emph{arbor polytopes} also have magic positive Ehrhart polynomials---this class of polytopes depends on two integer parameters and is not directly related to Pitman--Stanley polytopes.

One of the main motivations for proving the above result comes from a different encoding of Ehrhart polynomials. For any lattice polytope $\mathcal{P}$ of dimension $d$, one can find a polynomial $h^*_{\mathcal{P}}(x)$ of degree at most $d$ and satisfying:
    \begin{equation} \label{eq:hstar}
    \sum_{k\geq 0} \ehr_{\mathcal{P}}(k)\, z^k = \frac{h^*_{\mathcal{P}}(z)}{(1-z)^{d+1}}.
    \end{equation}
The polynomial appearing on the numerator of the right-hand-side is often called the \emph{Ehrhart $h^*$-polynomial}. It is a famous result by Stanley \cite{stanley-hstar} that the coefficients of $h^*_{\mathcal{P}}(z)$ are always nonnegative integers. Similar to how Ehrhart positivity is regarded as one of the most interesting classes of inequalities for the coefficients of $\ehr_{\mathcal{P}}(t)$, there is a hierarchy of attractive features that $h^*_{\mathcal{P}}(z)$ often displays: unimodality, log-concavity without internal zeros, or real-rootedness. Each of these three properties is stronger than the preceding ones (see \cite{stanley-unimodality}). We suggest \cite{braun-unimodality} and \cite{ferroni-higashitani} for a detailed discussion of these classes of inequalities in the specific context of Ehrhart $h^*$-polynomials. 

One of the many reasons why magic positivity is regarded as a very strong feature for an Ehrhart polynomial is that, via a remarkable result of Petter Br\"and\'en \cite{branden}, it implies that the $h^*$-polynomial of $\mathcal{P}$ is real-rooted (see \cite[Theorem~4.19]{ferroni-higashitani} for additional information). In other words, as a consequence of Theorem~\ref{thm:main-intro}, we obtain the following result.

\begin{theorem}\label{thm:main2-intro}
    Let $\Pi_n(\mathbf{y})$ be a Pitman--Stanley polytope. The polynomial $h^*_{\Pi_n(\mathbf{y})}(z)$ is real-rooted. In particular, its coefficients form a log-concave sequence without internal zeros and are unimodal.
\end{theorem}

Furthermore, motivated by famous conjectures on Ehrhart positivity of generalized permutohedra (a property disproved in \cite{ferroni} in general), we address the problem of magic positivity for a special class of generalized permutohedra called $\yy$-generalized permutohedra. Precisely, given a vector $\yy\in \mathbb{Z}^m_{\geq 0}$, and a bipartite graph $H\subseteq K_{m,n}$ with no isolated vertices, we consider the polytope
    \[ \mathcal{P}_H(y_1,\ldots,y_m) := y_1\Delta_{I_1} + \cdots + y_m\Delta_{I_m}, \]
where $\Delta_S$ denotes the convex hull of the canonical vectors $\{e_i: i\in S\}$, and where $I_i$ stands for the indices of all the neighbors of vertex $1\leq i\leq m$ in $H$. The polytopes $\mathcal{P}_{H}(\yy)$ arising in the form described above are called \emph{$\yy$-generalized permutohedra} and were introduced by Postnikov \cite{postnikov}, where he proved that they are Ehrhart positive (see Theorem~\ref{thm:post ehrhart Y-perm} below). It is not true, however, that these Ehrhart polynomials are always magic positive (cf. Example~\ref{ex:y perm not magic positive} appearing below). As it turns out, we prove that if the parameters $y_i$ of the vector $\yy$ are ``sufficiently large'' then $\ehr_{\mathcal{P}_H(\yy)}(t)$ is indeed magic positive. Precisely, we show the following.

\begin{theorem}\label{thm:main3-intro}
    Let $\yy = (y_1, \dots, y_m)$ be a vector of non-negative integers such that $y_i \geq n - 1$ for all $i$ and $y_1 \geq n$.
    Let $\yPerm_H({\bf y})$ be a $\mathbf{y}$-generalized permutohedron associated to a bipartite graph $H\subset K_{m,n}$.
    Then the Ehrhart polynomial $\ehr_{\yPerm_H({\bf y})}(t)$ is magic positive.
    Moreover, writing
    \[\mathfrak{D}_{\yy}(H) := \bigsqcup_{{\bf a}\in\Draco(H)}\ExtPerm{\bf a}{\yy},\]
    the coefficient of $t^j(1+t)^{n-1-j}$ equals, up to the normalization factor $(n-1)!$, the number of words in $\mathfrak{D}_{\yy}(H)$ with exactly $j$ lucky cars under the block parking protocol without the first available space.
\end{theorem}

We refer to the main body of the article for the undefined terms appearing in the above statement.
The most interesting part about this is not the magic positivity itself (because, in fact, it can also be deduced via elementary considerations on polynomials, see Proposition~\ref{prop:magic-y-perm-large}). The substance in the above statement is the fact that, also in this case, one has a combinatorial interpretation for the coefficients of the Ehrhart polynomial in the magic basis. As it turns out, in the main body of the article we need to prove this statement \emph{first} in order to come up with the right combinatorial interpretation that enables the proof of Theorem~\ref{thm:main-intro}.

\subsection*{Outline}
This article is organized as follows. Section~\ref{sec: background} has the background and notation used throughout. Section~\ref{sec. magic extenden} gives the general strategy for the proof of Theorem~\ref{thm:main-intro} and gives the proof of Theorem~\ref{thm:main3-intro} for ${\bf y}$-generalized permutohedra. Section~\ref{sec:magic} deals with the magic positivity and lucky statistic of ${\bf y}$-parking functions and finalizes the proof of Theorem~\ref{thm:main-intro}. Lastly, final remarks, examples of instances where magic positivity does not hold for related polytopes, and some open questions are given in Section~\ref{sec: final remarks}.

\subsection*{Related work}

An extended abstract \cite{avila-ferroni-morales-abstract} has been accepted for publication at the Proceedings of the FPSAC 2026 conference. Furthermore, in a companion paper \cite{avila-ferroni-morales-hstar} we provide a combinatorial interpretation for the $h^*$-polynomial of an arbitrary Pitman--Stanley polytope.

\subsection*{Acknowledgments}

The last two named authors would like to thank the American Institute of Mathematics (AIM) and the organizers of the AIM workshop {\em Ehrhart polynomials: inequalities and extremal constructions} in May 2022, where this project was started. We also thank  Christos Athanasiadis, William Dugan, Akihiro Higashitani, Annie Raymond, Richard Stanley, and Mei Yin for helpful comments and encouragement with this project.

\section{Background and notation} \label{sec: background}

We denote by $\bbinom{n}{k}$ the \textbf{multiset binomial coefficient}, counting the number of multisets of size $k$ that can be formed from a set of $n$ distinct elements and can be calculated by $$\bbinom{n}{k} = \binom{n + k - 1}{k}.$$

\subsection{\texorpdfstring{$\yy$}--Extended permutations}

In order to establish the main results of this article, we need to introduce some technical enumerative objects that we call ``extended permutations''.

\begin{definition}[$\yy$-extended permutations]

Let $\ss=(s_1,s_2,\dots,s_n)$ be a vector of nonnegative integers. We associate to $\ss$ the multiset $M_{\ss} := \{1^{s_1},2^{s_2},\dots,n^{s_n}\}$; that is, the multiset containing exactly $s_i$ copies of $i$, for each $i=1,\dots,n$. We denote by $\mulPerm{\ss}$ the set of all multipermutations of the multiset $M_{\ss}$; that is, all words of length $\sum_{i=1}^{n} s_i$ formed using the elements of $M_{\ss}$ respecting multiplicities.

Let $\yy=(y_1,\dots,y_n)$ be a vector of non-negative integers. Define a collection of sets $Y_1,\ldots,Y_n$ as follows: $Y_i = \left\{\left(\sum_{j=1}^{i-1} y_j\right) +1,\cdots,\sum_{j=1}^i y_j\right\}$. Let $\pi \in \mulPerm{\ss}$ be a multiset permutation, and let $r = \sum_{i=1}^n s_i$. We define a \textit{$\mathbf{y}$-extended permutation} associated to $\pi$ as any word $\sigma = (\sigma_1,\sigma_2,\dots,\sigma_r)$ where $\sigma_{j} \in Y_{\pi_{j}}$ for each $j\in [r]$. We define $\ExtPerm{\ss}{\yy}$ as the set of all \textit{$\mathbf{y}$-extended permutations} associated with the multiset permutations in $\mulPerm{\ss}$.
\end{definition}

\begin{example}\label{ej: y-permutation}
Let $\ss = (1, 2, 1)$. The multiset associated with $\ss$ is $M_{\ss} = \{1, 2, 2, 3\}$, and one possible permutation $\pi \in \mulPerm{\ss}$ is $\pi = (2, 3, 1, 2)$. Let us fix the vector $\mathbf{y}= (2,3,1)$ and let us construct all $\mathbf{y}$-extended permutations associated to $\pi$.

According to the definition we first identify the interval sets $Y_1,Y_2,Y_3$ induced by $\mathbf{y}$:
\[Y_1 = \{1,2\}, \quad Y_2 = \{3,4,5\}, \quad Y_3 = \{6\}.\]
Using the notation in the above definition, we have that $r = 1 + 2 + 1 = 4$. A $\mathbf{y}$-extended permutation associated to $\pi$ is thus a vector with four coordinates lying in $Y_2\times Y_3\times Y_1\times Y_2$. A direct computation shows that the number of $\mathbf{y}$-extended permutations is thus $|Y_2|\cdot |Y_3|\cdot |Y_1| \cdot |Y_2| = 3\cdot 1 \cdot 2 \cdot 3 = 18$. 
The following is an exhaustive list of all the eighteen $\mathbf{y}$-extended permutations associated to $\pi$:
\begin{center}
    \begin{tabular}{cccccc}
    $(3, 6, 1, 3)$ & $(3, 6, 1, 4)$ & $(3, 6, 1, 5)$ & $(3, 6, 2, 3)$ & $(3, 6, 2, 4)$ & $(3, 6, 2, 5)$ \\
    $(4, 6, 1, 3)$ & $(4, 6, 1, 4)$ & $(4, 6, 1, 5)$ & $(4, 6, 2, 3)$ & $(4, 6, 2, 4)$ & $(4, 6, 2, 5)$ \\
    $(5, 6, 1, 3)$ & $(5, 6, 1, 4)$ & $(5, 6, 1, 5)$ & $(5, 6, 2, 3)$ & $(5, 6, 2, 4)$ & $(5, 6, 2, 5)$ \\
    \end{tabular}
\end{center}
\end{example}

\begin{example}
Consider the one-coordinate vectors $\ss = (r)$ and $\yy = (m)$, with $r,m > 0$. In this case, we have $M_{\ss} = \{1^r\}$, and thus there is a single permutation $(1, 1, \dots, 1)$ of this multiset. Since each $1$ can be expanded in the interval $Y_1 = [m] = \{1, 2, \dots, m\}$, the associated $\yy$-extended permutations correspond to all words of length $r$ over the alphabet $[m]$. That is:
\[\ExtPerm{(r)}{(m)} = [m]^r.\]
\end{example}

\begin{example}
Consider $\yy = (1,\dots,1)\in \mathbb{Z}_{\geq 0}^n$. The intervals $Y_1,\ldots,Y_n$ consist only of one element $Y_{i} = \{i\}$. Thus, for any $\mathbf{s}=(s_1,\ldots,s_n)$ we have an identification between the sets $\ExtPerm{\ss}{\yy}$ and  $\mulPerm{\ss}$.
Furthermore, for $\ss = (1,\dots,1)\in \mathbb{Z}^n$, the associated multiset is $M_{\ss} = \{1,2,\dots,n\}$, and since each symbol appears exactly once, the set $\mulPerm{\ss}$ may be identified with the symmetric group $\mathfrak{S}_n$.
\end{example}

\begin{remark}
Let $r = \sum_{i=1}^n s_i$ be the size of permutations, then the total number of such permutations is given by the multinomial formula (see \cite[Eq. 1.22]{stanley-ec1}):
\begin{equation}\label{Eq. total_multpermutation}
    |\mulPerm{\ss}| = \binom{r}{s_1, s_2, \dots, s_n} = \frac{r!}{s_1! s_2! \cdots s_n!}
\end{equation}

More generally, for arbitrary $\yy = (y_1, \dots, y_n)$, in $\mulPerm{\ss}$ each symbol $i$ appears $s_i$ times, and since each symbol $i$ can be expanded in $y_i$ distinct ways, we obtain $y_i^{s_i}$ total expansions. Therefore, the total number of $\yy$-extended permutations in $\ExtPerm{\ss}{\yy}$ is:
\begin{equation}\label{Eq. total_y_permutation}
    |\ExtPerm{\ss}{\yy}| = \binom{r}{s_1, s_2, \dots, s_n} \cdot \prod_{i=1}^{n} y_i^{s_i} = r! \cdot \prod_{i=1}^{n} \frac{y_i^{s_i}}{s_i!}
\end{equation}
We adopt the convention that $0^0 = 1$, which reflects that when $s_i = 0$, the symbol $i$ does not appear in any permutation and its corresponding interval $Y_i$ is ignored.
\end{remark}

\subsection{y-parking functions}\label{subsec:y-pf}

The following generalization of the classical parking functions was defined in {\cite[Sec. 5]{Pitman_Stanley_1999}}.

\begin{definition}[$\yy$-parking functions]\label{def:y-parking-functions}
Given ${\bf y}=(y_1,\ldots,y_n)\in \mathbb{N}^n$, a ${\bf y}$-parking function is a sequence $\sigma=(s_1,\ldots,s_n)$ of positive integers such that their increasing rearrangement $b_1\leq \cdots \leq b_n$ satisfies $b_i\leq y_1+\cdots + y_i$. We denote this set by $\Pf_n({\bf y})$. The classical parking functions are the case ${\bf y}={\bf 1}$, that we denote by $\Pf_n:=\Pf_n({\bf 1})$.
\end{definition}

\begin{remark}
    These parking functions also appear in the literature as {vector parking functions} or as ${\bf u}$-parking functions \cite{KungYan,Yan_pf_handbook} where ${\bf u}=(u_1,\ldots,u_n)$ with $u_i=y_1+\cdots+y_i$.
\end{remark}

The $\yy$-parking functions  can be obtained from  classical parking functions  as follows in a process that we call {\em extension}. Given a parking function $\pi=(\pi_1,\ldots,\pi_n)$ in $\Pf_n({\bf 1})$, if we replace each $i$ in $\pi$ by an integer in $\{y_1+\cdots + y_{i-1}+1,\ldots, y_1+\cdots + y_i\}$, we obtain a $\yy$-parking function. Conversely, every $\yy$-parking function can be uniquely obtained this way. This gives the following identity for the number of $\yy$-parking functions in terms of the classical parking functions.

\begin{proposition}[{\cite[Thm. 11]{Pitman_Stanley_1999}}] \label{prop:decompression}
For non-negative integers $\yy=(y_1,\ldots,y_n)$ we have that 
\[
\# \Pf_n(\yy)\,=\, \sum_{\pi \in \Pf_n({\bf 1})} y_{\pi_1}\cdots y_{\pi_n}.
\]
\end{proposition}

\begin{example} \label{ex: y-parking}
For $\yy=(2,1)$, the $\yy$-parking functions are $$\{(1,1),(1,2),(2,1),(2,2)\} \cup \textcolor{blue}{\{(1,3), (2,3)\}} \cup \textcolor{red}{\{(3,1), (3,2)\}},$$ which are obtained by extending the classical parking functions $(1,1),\textcolor{blue}{(1,2)},\textcolor{red}{(2,1)}$. Moreover, $\#\Pf_n(\yy)=y_1^2+\textcolor{blue}{y_1y_2} + \textcolor{red}{y_2y_1}$. 
\end{example}

\subsection{Pitman--Stanley polytopes and generalized permutohedra}

Postnikov defined in \cite{postnikov} the following polytopes called $\mathbf{y}$-generalized permutohedra. These polytopes form a strict but very large subclass of generalized permutohedra. For $I\subset [n]$, let $\Delta_I=\operatorname{conv}(e_i \mid i\in I)$ denote the coordinate simplex indexed by $I$. Given a bipartite graph $H\subset K_{m,n}$ with no isolated vertices and ${\bf y}=(y_1,\ldots,y_m)\in \mathbb{Z}^m_{\geq 0}$, let $\yPerm_H({\bf y})$ be the following Minkowski sum
\[
\yPerm_H(y_1,\ldots,y_m) = y_1\Delta_{I_1} + \cdots + y_m \Delta_{I_m},
\]
where $I_i=\{j \mid (i,j) \in E(H)\}$   for $i=1,\ldots,m$.

An {\em $H$-Draconian sequence} is a tuple $\aaa=(a_1,\ldots,a_m)$ of non-negative integers $(a_1,\ldots,a_m)$ such that $\sum_i a_i = n-1$ and for any subset $\{j_1,\ldots,j_k\}\subset [m]$ we have that $|I_{j_1} \cup \cdots \cup I_{j_k}|\geq a_{j_1}+\cdots + a_{j_k}+1$. Let $\Draco(H)$ be the set of $H$-Draconian sequences. 

Postnikov gave in \cite{postnikov} the following formula for the Ehrhart polynomial of $\yPerm_H({\bf y})$.

\begin{theorem}[{\cite[Thm. 11.3]{postnikov}}] \label{thm:post ehrhart Y-perm}
\[
\ehr_{\yPerm_H({\bf y})}(t) = \sum_{{\bf a}\in\Draco(H)} \bbinom{t \cdot y_1+1}{a_1} \ \bbinom{t\cdot y_2}{a_2}\cdots \bbinom{t \cdot y_m}{a_m}.
\]    
\end{theorem}

This formula makes clear that the polynomial $\ehr_{\yPerm_H({\bf y})}(t)$ has nonnegative coefficients, because each factor appearing in each summand has nonnegative coefficients. This phenomenon does not extend to all generalized permutohedra (see \cite{ferroni} for several examples).

Postnikov showed in \cite{postnikov} that Pitman--Stanley polytopes are an example of  a $\mathbf{y}$-generalized permutohedron.

\begin{proposition}[{\cite[Ex. 9.7]{postnikov}}]
Let ${\bf y}=(y_1,\ldots,y_n) \in \mathbb{Z}_{\geq 0}^n$, and $H\subset K_{n,n+1}$ be defined by the sets $I_i=[n+2-i]$ for $i=1,\ldots,n$, then $P_H({\bf y})$ is the Pitman--Stanley polytope $\Pi_n({\bf y})$.
\end{proposition}

Since $\Pi_n({\bf y})$ is  a $\yy$-generalized permutohedron, as a consequence one may specialize Proposition~\ref{thm:post ehrhart Y-perm} to obtain a nonnegative formula for the Ehrhart polynomial of Pitman-Stanley polytopes. Such specializations appear already in the work of Pitman and Stanley. In order to state these formulas without resorting to the language of draconian sequences, we introduce the set $\mathcal{I}_n$ of all weak compositions of $n$ that are $\geq {\bf 1}$ in dominance order. That is

\begin{equation}\label{Eq. parking over compositions}
    \mathcal{I}_n:=\left\{(s_1,\ldots,s_n)\in \mathbb{Z}_{\geq 0}^n \; : \;\sum_{j=1}^n s_j = n \;\text{ and }\; \sum_{j=1}^i s_j \geq i, \text{  for each }i=1,\ldots,n-1\right\}.
\end{equation}
Note that these compositions are among the many objects counted by the Catalan numbers $C_n=\frac{1}{n+1}\binom{2n}{n}$ \cite[2.87]{CatalanBook}.

\begin{corollary}[{Pitman--Stanley \cite[Eq. (33)]{Pitman_Stanley_1999}}]\label{coro:ehr-PS}
Let ${\bf y}=(y_1,\ldots,y_n) \in \mathbb{Z}_{\geq 0}^n$, then
\begin{align}
\ehr_{\Pi_n({\bf y})}(t) &= 
 \sum_{{\bf s} \in \mathcal{I}_n} \bbinom{t\cdot y_1+1}{s_1} \bbinom{t\cdot y_2}{s_2}\cdots \bbinom{t\cdot y_n}{s_n}. \label{eq: catalana formual lattice points}
\end{align}
\end{corollary}

\begin{example} \label{ex:Ehr PS2}
For $n=2$, we have that $I_2$ consists of the compositions $(1,1)$ and $(2,0)$. Thus \eqref{eq: catalana formual lattice points} gives $$\ehr_{\Pi_2(y_1,y_2)}(t)= \bbinom{ty_1+1}{1}\bbinom{ty_2}{1} + \bbinom{ty_1+1}{2}=\frac{1}{2}\left( y_{1}^{2}+ 2y_{1} y_{2}\right)t^{2}+\left(\frac{3}{2}  y_{1}+ y_{2}\right)t+1.$$ For $(y_1,y_2)=(2,1)$, this gives $\ehr_{\Pi_2(2,1)}(t)=4t^2+4t+1$.
\end{example}

For general ${\bf y}$, there is a determinantal formula for the Ehrhart polynomial (see \cite[Sec. 5, Eq. (25)]{Pitman_Stanley_1999}) that comes from a plane-partition interpretation of the lattice points of this polytope (see Section~\ref{sec:magic kreweras determinants}).

\begin{theorem}[Pitman--Stanley]
For the Pitman--Stanley polytope $\Pi_n(y_1,\ldots,y_n)$ we have that 
\begin{align}
\ehr_{\Pi_n({\bf y})}(t) 
&\,=\, \det\left[\binom{t\cdot(y_1+\cdots +y_{i})+1}{j-i+1}\right]_{i,j=1}^n. \label{eq:detEhrhartPS}
\end{align}
\end{theorem}

For certain values of ${\bf y}$ there is a product formula for the Ehrhart polynomial.

\begin{theorem}[{\cite[Thm. 13]{Pitman_Stanley_1999}}]
When ${\bf y} = (a,b,b,\ldots,b) \in \mathbb{Z}^n_{\geq 0}$ we have that 
\[
\ehr_{\Pi_n({\bf y})}(t) = \frac{1}{n!} (ta+1)(t(a+bn)+2)(t(a+nb)+3)\cdots ( t(a+nb)+n).
\]
In particular, for ${\bf y} = {\bf 1}$ we have that  $\ehr_{\Pi_n({\bf 1})}(t) = \frac{1}{t(n+1)+1}\binom{(t+1)(n+1)}{n+1}$.
\end{theorem}

\begin{remark}\label{remk: y-parking to y-extended}
    There is an interpretation for the normalized volume of $\Pi_n({\bf y})$ by enumerating $\yy$-parking functions. This is because by Definition~\ref{def:y-parking-functions}, the set of $\yy$-parking functions $\Pf_n(\yy)$  can be parameterized through a family of weak compositions that dominate a uniform slope $\mathcal{I}_{n}$ and $\yy$-extended permutation.
    
    Precisely, we have that a $\yy$-parking function $\sigma \in \Pf_n(\yy)$ is a $\yy$-extended permutation in $\ExtPerm{\ss}{\yy}$, for some $s \in \mathcal{I}_{n}$ and therefore:
    \[\Pf_n(\yy) := \bigsqcup_{\mathbf{s} \in \mathcal{I}_{n}} \mathfrak{S}_{\mathbf{s}, \mathbf{y}}.\]
\end{remark}

For certain values of ${\bf y}$, there are nice formulas for $\#\Pf_n(\bf y)$  (see \cite{Yan_pf_handbook}).  For example, $\#\Pf_n({{\bf 1}})=(n+1)^{n-1}$, the number of classical parking functions.

\begin{theorem}[{\cite[Thm. 11]{Pitman_Stanley_1999}}] \label{thm:vol PS in terms of PF}
For ${\bf y} \in \mathbb{N}^n$ we have that 
\[
\vol \Pi_n(\yy) = \#\Pf_n(\yy).
\]
In particular, we have that $\vol \Pi_n({\bf 1})= (n+1)^{n-1}$.
\end{theorem}

\subsection{Real-rootedness and magic basis}

Let us denote by $\mathbb{R}[t]_d$ the set of real polynomials of degree $d$. Let $p(t) \in \mathbb{R}[t]_d$, say $p(t) = a_d\, t^d + \ldots + a_1 t + a_0$. We can write $p(t)$ in an alternative basis, $\{t^i(1+t)^{d-i}\}_{i=0}^d$, thus obtaining
    \[ p(t) = \sum_{i=0}^d c_i\, t^i(1+t)^{d-i}\]
for some real numbers $c_0,\ldots,c_d$.
We call $c_0,c_1,\ldots,c_d$ the \emph{magic coefficients} of $p(t)$ and we say that $p(t)$ is \emph{magic positive} whenever $c_0,\ldots,c_d\geq 0$. We have a linear map $\mathscr{M}_d : \mathbb{R}[t]_d \to \mathbb{R}[t]$ defined on $p(t)$ (given as above) by the formula
    \[ \mathscr{M}_d(p) = \sum_{j=0}^d c_j\, t^j.\]
Note that even though $p(t)$ has degree $d$, the degree of $\mathscr{M}_d(p)$ is only \emph{at most} $d$. The following is a useful result that we will employ several times in the sequel.

\begin{lemma}\label{lemma:magic-multiplicative}
    Let $p(t) \in \mathbb{R}_d[t]$ and $q(t)\in \mathbb{R}_e[t]$. Then:
    \[ \mathscr{M}_{d+e}(p\cdot q) = \mathscr{M}_{d}(p) \cdot \mathscr{M}_{d}(q)\]
\end{lemma}

\begin{proof}
    The result follows from the next observation: the polynomial $\mathscr{M}_d(p)$ can be obtained as the composition of the following three operations: (i) first reverse the coefficients of $p(t)$, (ii) make the change of variables $t \mapsto t-1$, (iii) reverse the coefficients again.
\end{proof}

As a consequence of the above result, magic positivity is closed under products. From now on, when we write $\mathscr{M}(p)$ without specifying any subindex in the operator $\mathscr{M}$, we will be implicitly assuming that the subindex is exactly the degree of the polynomial $p$. With these conventions, the last result is simply saying that $\mathscr{M}(pq) = \mathscr{M}(p) \cdot \mathscr{M}(q)$. 

\begin{corollary}\label{coro:magic-binomials}
    Let $a,b,c\in \mathbb{Z}_{\geq 0}$, and consider the polynomial \(p_{a,b,c}(t) := \binom{at+b}{c}.\)
    Then, the following holds:
    \( \mathscr{M}(p_{a,b,c})(t) = \frac{1}{c!}\prod_{i=0}^{c-1} (b-i + (a+i-b)t).\)
    In particular, if $a\geq b\geq c-1$ then $p_{a,b,c}$ is magic positive.
\end{corollary}

\begin{proof}
    Since $p_{a,b,c}(t) = \frac{1}{c!} \prod_{i=0}^{c-1} (at+b-i)$, thanks to the preceding lemma we have the following chain of equalities:
    \begin{align*}
        c!\mathscr{M}(p_{a,b,c}) &= \prod_{i=0}^{c-1} \mathscr{M}(at+b-i)\\
        &= \prod_{i=0}^{c-1} \mathscr{M}((b-i)t^0(1+t)^{1-0} + (a+i-b)t^1(1+t)^{1-1})\\
        &= \prod_{i=0}^{c-1} (b-i + (a+i-b)t).
    \end{align*}
    If $a\geq b\geq c-1$ then each of the numbers $b-i$ and $a+i-b$ is nonnegative for $0\leq i\leq c-1$.
\end{proof}

\section{Magic positivity and lucky statistic on \texorpdfstring{$\yy$}--extended permutations}\label{sec. magic extenden}

The main goal in this section is to pave the way towards proving Theorem~\ref{thm:main-intro}, and we will in fact prove Theorem~\ref{thm:main3-intro}. We first start with a sequence of useful lemmas and observations that motivate the core of the strategy. 

\subsection{First steps into magic positivity}

The first ingredient is the following closed formula for the Ehrhart polynomial of $\Pi_n(\mathbf{y})$ in the ``magic basis''.

\begin{lemma}\label{lem. trasformacion polynomiasl}
    Let $\mathbf{y}\in \mathbb{Z}^n_{\geq 0}$. Then, we have:
    \begin{equation}\label{eq:magic-ehr-ps} 
    \mathscr{M}(\ehr_{\Pi_n(\mathbf{y})})(t) = \frac{1}{n!}\sum_{{\bf s} \in \mathcal{I}_n}\binom{n}{s_1,\dots,s_n}\prod_{j=1}^{s_1} (j + (y_1-j)t) \prod_{i=2}^n \prod_{j=1}^{s_i} (j-1 + (y_i-j+1)t).
\end{equation}
\end{lemma}

\begin{proof}
    We employ Corollary~\ref{coro:ehr-PS} together with Lemma~\ref{lemma:magic-multiplicative} and Corollary~\ref{coro:magic-binomials}. The first observation is that all the summands in the first formula of Corollary~\ref{coro:ehr-PS} have the same degree, so that $\mathscr{M}$ acts linearly on the sum, and we obtain the following chain of equalities:
    \begin{align*}
        \mathscr{M}(\ehr_{\Pi_n(\mathbf{y})})(t) &= \sum_{{\bf s} \in \mathcal{I}_n} \mathscr{M}\left(\bbinom{t\cdot y_1+1}{s_1} \bbinom{t\cdot y_2}{s_2}\cdots \bbinom{t\cdot y_n}{s_n}\right)\\
        &= \sum_{{\bf s} \in \mathcal{I}_n} \mathscr{M}\left(\binom{t\cdot y_1+s_1}{s_1} \binom{t\cdot y_2+s_2-1}{s_2}\cdots \binom{t\cdot y_n+s_n-1}{s_n}\right)\\
        &= \sum_{{\bf s} \in \mathcal{I}_n} \mathscr{M}\binom{t\cdot y_1+s_1}{s_1} \prod_{i=2}^n \mathscr{M}\binom{t\cdot y_i+s_i-1}{s_i}\\
        &= \sum_{{\bf s} \in \mathcal{I}_n} \frac{1}{s_1!}\prod_{j=0}^{s_1-1} (s_1 - j + (y_1-s_1+j)t) \prod_{i=2}^n \frac{1}{s_i!}\prod_{j=0}^{s_i-1} (s_i-1 - j + (y_i-s_i+1+j)t)\\
        &= \frac{1}{n!}\sum_{{\bf s} \in \mathcal{I}_n}\binom{n}{s_1,\dots,s_n}\prod_{j=1}^{s_1} (j + (y_1-j)t) \prod_{i=2}^n \prod_{j=1}^{s_i} (j-1 + (y_i-j+1)t),
    \end{align*}
    where in the last step we applied the change of variable $j \mapsto s_i-j$ for each product.
\end{proof}

\begin{example} \label{ex:PS 0 ys that is not magic positive}
If one removes the positivity assumption on $\yy\in \mathbb{Z}^n_{>0}$ for the Ehrhart polynomial of $\Pi_n(\yy)$, then Theorem~\ref{thm:main-intro} need not hold. For example, for $n=3$ and $\yy=(1,0,2)$ for $\Pi_3(\yy)$, the Ehrhart polynomial is $\frac{7}{6} t^{3} + 4 t^{2} + \frac{23}{6} t + 1$
 which is not magic positive, because:
\[ \mathscr{M}(\ehr_{\Pi_3(\yy)})(t) = 
- \frac{2}{3} t^2 + \frac{5}{6} t + 1 
.\] 
This example showcases a difficulty we should overcome to prove positivity under the assumption that $\yy\in\mathbb{Z}^n_{>0}$.
\end{example}

We will henceforth use the following notation: \[{f}_{\Pf_n(\yy)}(t)=n!\cdot \mathscr{M}(\ehr_{\Pi_n(\mathbf{y})})(t).\]

Thanks to Lemma~\ref{lem. trasformacion polynomiasl}, our Theorem~\ref{thm:main-intro} would follow if we can prove that $f_{\Pf_n(\yy)}(t)$ has nonnegative coefficients whenever $\yy\in\mathbb{Z}^n_{>0}$. As the reader will learn in the sequel, the proof of this positivity phenomenon is quite subtle under the general assumption that $\yy\in \mathbb{Z}^n_{>0}$. However, for certain choices of $\yy$ we can conclude in a straightforward way that $f_{\Pf_n(\yy)}(t)$ has nonnegative coefficients.

\begin{proposition} \label{prop: PS for large y}
    Let $\yy\in \mathbb{Z}^n_{>0}$ be a vector of positive integers such that $y_1\geq n$ and $y_i \geq n - 1$ for each $i=2,\ldots,n$. Then, all the summands appearing on the right hand side of equation \eqref{eq:magic-ehr-ps} have nonnegative coefficients. In particular, for such choices of $\yy$ the Ehrhart polynomial of the Pitman--Stanley polytope $\Pi_n(\yy)$ is magic positive.
\end{proposition}

\begin{proof}
    Since the set $\mathcal{I}_n$ consists of compositions of $n$, we have the trivial inequalities $0\leq s_i \leq n$ for each $1\leq i\leq n$. In particular, if $y_1\geq n$ we have that $y_1\geq s_1$ and so $y_1-j\geq 0$ for each $1\leq j\leq s_1$. Similarly, for $2\leq i \leq n$, the condition $y_i\geq n-1$ implies that $y_i + 1 \geq n \geq s_i$ and so $y_i-j+1\geq 0$ for $1\leq j \leq s_i$. 
\end{proof}

An alternative way of proving the preceding proposition follows from the fact that the polynomials $p_{a,b,c}(t) = \binom{at+b}{c}$ are magic positive when $a\geq b\geq c-1$ (as was shown in Corollary~\ref{coro:magic-binomials}). This, combined with Corollary~\ref{coro:ehr-PS} gives the desired positivity. In fact, this idea can also be applied for some $\mathbf{y}$-generalized permutohedra. Note that the next restriction on the parameters $\yy$ is off by one from the case of $\Pi_n(\yy)$ from Proposition~\ref{prop: PS for large y}. This is because the way we defined these two polytopes, $\Pi_n(\yy)$ is full dimensional while $\mathcal{P}_H(\yy)$ has codimension one.

\begin{proposition}\label{prop:magic-y-perm-large}
    Let $\yy\in \mathbb{Z}^n_{> 0}$ be such that $y_1 \geq n - 1$ and $y_i \geq n-2$ for $i=2,\ldots,n$. Then, the Ehrhart polynomial of $\mathcal{P}_H(\yy)$ is magic positive.
\end{proposition}

\begin{proof}
    We can use the formula in Theorem~\ref{thm:post ehrhart Y-perm}. All the summands have the same degree, so it suffices to show that each summand is magic positive. Using the notation $p_{a,b,c}(t) := \binom{at+b}{c}$, the factors in each summand are $p_{y_1,a_1,a_1}(t)$ and $p_{y_i,a_i-1,a_i}(t)$ for $2\leq i\leq n$. By Corollary~\ref{coro:magic-binomials}, the conditions $y_1 \geq a_1$ and $y_i \geq a_i-1$ for $i=2,\ldots,n$ are enough to conclude that each single factor is magic positive. This follows readily from the assumption that $y_1 \geq n-1$ and $y_i\geq n-2$ for $i=2,\ldots,n$.
\end{proof}

There is no hope of extending the preceding theorem to all $\yy$-generalized permutohedra, as they fail to have magic positive Ehrhart polynomials in general. See Example~\ref{ex:y perm not magic positive}.

\begin{guide*}{}{}\label{guide}
The overall strategy to prove Theorem~\ref{thm:main-intro} is:
\begin{enumerate}[Step (i)]
    \item We show that for $\yy\in \mathbb{Z}^n_{> 0}$ such that $y_1 \geq n - 1$ and $y_i \geq n-2$ for $2\leq i \leq n$, there is a combinatorial interpretation for the coefficients of the polynomial $\mathscr{M}(\ehr_{\mathcal{P}_H(\yy)})$.
    \item We specialize the interpretation of Step (i) to the polynomials $\mathscr{M}(\ehr_{\Pi_n(\yy)})$ for $\yy\in\mathbb{Z}^n_{>0}$ such that $y_1\geq n$ and $y_i\geq n-1$ for $2\leq i \leq n$.
    \item We show that the combinatorial interpretation for $\mathscr{M}(\ehr_{\Pi_n(\yy)})$ in Step (ii) continues to hold, after some adjustments, for all $\yy\in \mathbb{Z}^n_{>0}$. 
\end{enumerate}
\end{guide*}

\subsection{Algebraic contributions}

Throughout this section we will always assume that $\ss, \yy \in \Z_{\geq 0}^{n}$ are vectors of non-negative integers, and that $r:=\sum_{i=1}^n s_i$ (very often we will find ourselves in the situation that $r=n$, but this is not strictly required). For each integer $0\leq l\leq n$ we define the polynomial:
\begin{equation}\label{eq. polynomial of algebraic lucky in}
    {f}_{\ss,l}(t,\yy) := \binom{r}{s_1,\dots,s_n}\prod_{i=1}^{n}\prod_{j=1}^{s_i}\left((y_{i} - (j-1 + \delta_{i,l}))t + (j-1 + \delta_{i,l})\right).
\end{equation}

In the above display (and throughout the remainder of this paper) the notation $\delta_{i,j}$ stands for Kronecker's delta, i.e., $\delta_{i,j} = 0$ for $i\neq j$ and $\delta_{i,i} = 1$. Note that the cases $l=0$ and $l=1$ of the above polynomial are especially relevant:
\begin{align*}
    {f}_{\ss,0}(t,\yy) &= \binom{r}{s_1,\dots,s_n}\prod_{i=1}^n \prod_{j=1}^{s_i} (j-1 + (y_i-j+1)t),\\
    {f}_{\ss,1}(t,\yy) &= \binom{r}{s_1,\dots,s_n}\prod_{j=1}^{s_1} (j + (y_1-j)t) \prod_{i=2}^n \prod_{j=1}^{s_i} (j-1 + (y_i-j+1)t),
\end{align*}

By combining equations~\eqref{eq:magic-ehr-ps} and~\eqref{eq. polynomial of algebraic lucky in} we can write:
\begin{equation} \label{eq:refined}
{f}_{\Pf_n(\yy)}(t) = \sum_{\mathbf{s} \in \mathcal{I}_{n}} {f}_{\ss,1}(t,\yy).
\end{equation}

The next is a very technical but crucial definition in this article.

\begin{definition}[Algebraic contribution]\label{def. algebraic contributions}
    Let us fix $l=0,1,\ldots,n$ and let $\pi \in \mulPerm{\ss}$ be a multipermutation.
    For each $1\leq j\leq r$, we define $\mu(\pi)_{j}$, called the \emph{$j$-th partial multiplicity} of $\pi$ as the number:
    \begin{equation}\label{eq. partial mutiplicity}
        \mu(\pi)_{j} := \#\{ k \leq j : \pi_k = \pi_j \}.
    \end{equation}
    Let us also fix the notation:
    \begin{equation}\label{eq. occupied spaces}
        o_{j,l}(\pi) := \mu(\pi)_{j} - 1 +\delta_{\pi_j,l}.
    \end{equation}
    Fix $L\subseteq [r]$. The \textit{algebraic contributions of type $\alpha$ and $\beta$} of the quadruple $(\yy,l,\pi,L)$ are defined, respectively, by: 
    \begin{align*}
        \alg_{l}^{(\alpha)}(\yy;\pi,L) &:= \prod_{j \in L}y_{\pi_j} \cdot \prod_{j \in ([r]\setminus L)}o_{j,l}(\pi).\\
        \alg_{l}^{(\beta)}(\yy;\pi,L) &:= \prod_{j \in L}( y_{\pi_j} - o_{j,l}(\pi)) \cdot \prod_{j \in ([r]\setminus L)}o_{j,l}(\pi).
    \end{align*}
\end{definition}

The names ``type $\alpha$'' and ``type $\beta$'' resemble the classical notation used to denote the flag $f$-vector and the flag $h$-vector of a graded poset (see \cite[Section~3.13]{stanley-ec1}). In particular, as the following lemma illustrates, the relationship between these two types of contributions parallels the linear relationship between the flag $f$ and the flag $h$-vector of such posets.

\begin{lemma}\label{lemma. relation into beta and alpha alg}
     The algebraic contributions of types $\alpha$ and $\beta$ of the quadruple $(\yy,l,\pi,L)$ are related via the following:
     \begin{enumerate}[\normalfont (i)]
         \item $\alg_{l}^{(\alpha)}(\yy;\pi,L) = \displaystyle\sum_{S \subseteq L}\alg_{l}^{(\beta)}(\yy;\pi,S)$ \\ 
         \item $\alg_{l}^{(\beta)}(\yy;\pi,L) = \displaystyle\sum_{S \subseteq L}(-1)^{|L| - |S|}\alg_{l}^{(\alpha)}(\yy;\pi,S)$
     \end{enumerate}
\end{lemma}

\begin{proof}
   Identity (ii) follows from (i) by the inclusion–exclusion principle. To prove (i), observe that
    \begin{align*}
        \sum_{S \subseteq L}\alg_{l}^{(\beta)}(\yy;\pi,S) &= \prod_{j \in ([r] \setminus L)}o_{j,l}(\pi) \cdot  \sum_{S \subseteq L} \left(\prod_{j \in S}(y_{\pi_j} - o_{j,l}(\pi)) \cdot \prod_{j \in (L\setminus S)}o_{j,l}(\pi)\right)
    \end{align*}
    Expanding the inner sum, we obtain
    \begin{align*}
        \sum_{S \subseteq L}\alg_{l}^{(\beta)}(\yy;\pi,S) &= \prod_{j \in ([r] \setminus L)}o_{j,l}(\pi) \cdot \prod_{j \in L}(y_{\pi_j} - o_{j,l}(\pi) + o_{j,l}(\pi)) \\
        &= \prod_{j \in ([r] \setminus L)}o_{j,l}(\pi) \cdot \prod_{j \in L} y_{\pi_j} = \alg_{l}^{(\alpha)}(\yy;\pi,L). \qedhere
    \end{align*}
\end{proof}

As we shall now demonstrate, the definition of algebraic contribution plays the role of a convenient way of refining the polynomial ${f}_{\ss,l}(t,\yy)$ by using multipermutations $\pi \in \mulPerm{\ss}$ and subsets $L \subseteq [r]$ where $r = \sum_{i=1}^{n}s_i$. Let us fix $\pi \in \mulPerm{\ss}$, and define the following polynomial: 
\begin{equation}
    {f}_{\pi,l}(t,\yy) := \sum_{L \subseteq [r]}\alg_{l}^{(\beta)}(\yy;\pi, L)t^{|L|}
\end{equation}

\begin{lemma}
The polynomial $f_{\pi,l}(t,\yy)$ defined above does not depend on the multipermutation $\pi$. Moreover, it is given by:
\begin{equation}\label{eq. produt f_pi_y not depends}
    f_{\pi,l}(t,\mathbf{y})
    = \prod_{i=1}^{n}\prod_{k=1}^{s_i}
    \Big((y_i-(k-1+\delta_{i,l}))t + (k-1+\delta_{i,l})\Big).
\end{equation}
\end{lemma}

\begin{proof}
Expanding the sum and by the definition of the product of $\alg_{l}^{(\beta)}(\yy;\pi, L)$, we have that
\begin{equation}\label{eq. produt f_pi_y}
    {f}_{\pi,l}(t,\yy) = \prod_{j = 1}^{r}\Big((y_{\pi_j} - o_{j,l}(\pi))t + o_{j,l}(\pi)\Big)
\end{equation}

Now let us consider the following set of tuples. Let $\ss = (s_1,\dots,s_n) \in \Z_{\geq 0}^{n}$, then we denote the set 
$$\mathcal{T}_{\ss}: = \{(i,k) : 1 \leq i \leq n \text{ and } 1 \leq k \leq s_i\}.$$

Note that the map $\psi_{\ss} :\mathcal{T}_{\ss} \longrightarrow [r]$, where $\psi_{\ss}(i,k) = \sum_{j=1}^{i-1}s_j + k$, is a bijection between the elements of $\mathcal{T}_{\ss}$ and the index of $[r]$, so, for all $\pi \in \mulPerm{\ss}$.

Recall that for a multipermutation $\pi \in \mulPerm{\mathbf{s}}$, the quantity $o_{j,l}(\pi) := \mu(\pi)_{j} - 1 +\delta_{\pi_j,l}$, depends only on two pieces of information: the symbol $i=\pi_j$ appearing at position $j$, and the partial multiplicity $\mu(\pi)_{j}$ which records how many times the symbol $i$ has appeared up to position $j$. For a fixed symbol $i$, the values of $\mu(\pi)_j$ as $j$ ranges over the positions where $\pi_j=i$ are exactly the integers $1,2,\dots,s_i$, in some order depending on $\pi$. Thus, for each occurrence of $i$, there is a unique index $k\in\{1,\dots,s_i\}$ such that $\mu(\pi)_j = k$. Substituting this into the definition of $o_{j,l}(\pi)$, we obtain  
\(o_{j,l}(\pi)=k-1+\delta_{i,l}.\)
 This shows that the factor  
\(\big((y_{\pi_j}-o_{j,l}(\pi))t + o_{j,l}(\pi)\big)\)
depends only on the pair $(i,k)$ associated to the $k$-th occurrence of the symbol $i$, and not on the specific position $j$ of the occurrence in the multipermutation.

Since multiplication is commutative, the order in which these factors appear is irrelevant. Therefore, the product over positions $j=1,\dots,r$ can be rewritten as a product over all pairs $(i,k)\in\mathcal{T}_{\mathbf{s}}$:
\begin{align*}
    f_{\pi,l}(t,\mathbf{y})
    &= \prod_{j=1}^{r}\Big((y_{\pi_j}-o_{j,l}(\pi))t + o_{j,l}(\pi)\Big) \\
    &= \prod_{i=1}^{n}\prod_{k=1}^{s_i}
    \Big((y_i-(k-1+\delta_{i,l}))t + (k-1+\delta_{i,l})\Big),
    \end{align*}
which is of course manifestly independent of the choice of multipermutation $\pi$.
\end{proof}

\begin{proposition}\label{prop. refine of f_s,y by pi and L}
     Let $\ss = (s_1,\dots,s_n)$ and $\yy = (y_1,\dots,y_n)$ be vectors of non-negative integers, the polynomial ${f}_{\ss,l}(t,\yy)$ can be refined in terms of subsets $L \subseteq [r]$ and multipermutations $\pi \in \mulPerm{\ss}$,
     \[{f}_{\ss,l}(t,\yy) = \sum_{L \subseteq [r]}\sum_{\pi \in \mulPerm{\ss}}\alg_{l}^{(\beta)}(\yy;\pi , L)t^{|L|}.\]
\end{proposition}

\begin{proof}
    The result follows directly from equation~\eqref{eq. produt f_pi_y not depends}. Since it does not depend on $\pi$, summing over all the multipermutations in $\mulPerm{\ss}$, we obtain
    \begin{equation*}
        \sum_{\pi \in \mulPerm{\ss}}{f}_{\pi,l}(t,\yy) = \binom{r}{s_1,\dots,s_n}\prod_{i=1}^{n}\prod_{k = 1}^{s_i}\Big((y_{i} - (k-1+\delta_{i,l}))t + (k-1+\delta_{i,l})\Big),
    \end{equation*}
    which is ${f}_{\ss,l}(t,\yy)$ by definition.
\end{proof}

\subsection{Block parking protocol}\label{subsec:block-parking-protocol}

The goal in this subsection is to complete the first two steps of the overall strategy described in the guide presented earlier in this section. Recall that a \emph{parking function} of order $n$ is a word $p=(p_1,\dots,p_n)$ of integers $p_1,\ldots,p_n \in [n]$, satisfying that, for each $j\in [n]$, there are at least $j$ values that are at most $j$. This notion is often visualized as follows: there are $n$ cars and $n$ parking spaces numbered from $1$ to $n$; the $i$-th car has a preferred parking space $p_i$. The cars enter the parking lot in order, and the $i$-th car goes directly to the space in position $p_i$. If it finds it free then it parks there, and otherwise it continues driving until the next available space. A parking function corresponds to a list of preferred spaces that allows all the cars to park.

Inspired by this ``parking protocol'' (i.e., this list of rules the cars follow in order to park), we now introduce a block version that interacts with the structure of $\yy$-extended permutations. 
Given a vector $\yy=(y_1,\dots,y_n)$ of non-negative integers, recall the associated intervals
\begin{equation}\label{eq:Y-interval}
Y_i = \left[\sum_{k=1}^{i-1}y_k + 1, \sum_{k=1}^{i}y_k\right], \qquad i \in [n].
\end{equation}

\begin{protocol*}{Block Parking Protocol}{}
Fix $r, n\geq 1$ and let $\yy\in \mathbb{Z}_{\geq 0}^n$. There are $r$ cars that desire to park, and there are $m := \sum_{i=1}^n y_i$ parking spaces. Moreover, the parking spaces are divided into $n$ circles, each containing (in cyclic order) the parking spaces labelled by the intervals $Y_1,\ldots, Y_n$ appearing in equation~\eqref{eq:Y-interval}.

Let us fix $w = (w_1,\dots,w_r) \in [m]^r$ a list of preferences for the $r$ cars, and let $l\in\{0,1,\dots,n\}$ be a fixed parameter. If $l>0$, the first position of block $Y_l$ is declared permanently unavailable; if $l=0$, all spaces are available.

Cars attempt to park, in order, according to the following rules.  
If $w_j\in Y_i$, the $j$-th car tries to park at $w_j$. If that space is unavailable or already occupied, the car advances cyclically within $Y_i$, until it finds an available space; if it cannot park, then it leaves.

We denote by $\out_{\circlearrowright,l}(\yy;w)=(b_1,\dots,b_r)$ the output configuration, where $b_j$ is the space where the $j$-th car parks, and $b_j=\varnothing$ if the car fails to park.
\end{protocol*}
We illustrate the block parking protocol through two basic examples.
\begin{example}
    Let $n = 3$ and $r = 14$. Consider $\yy = (3, 9, 6)$, so that $m = 3+9+6 = 18$. Fix $l = 0$, so that all spaces are indeed available. Let us fix the list of preferences in $w\in [18]^{14}$ given by 
    $w = (5,17,17,3,8,18,3,1,8,13,11,2,7,11)$. In Figure \ref{Fig. example of parking blocks} we illustrate the outcome of the block parking protocol for these $14$ cars.
    
    As a result we obtain $\out_{\circlearrowright,0}(\yy;w) = (5, 17, 18, 3, 8, 13, 1, 2, 9, 14, 11,\varnothing, 7, 12)$,  where the $12$-th car fails to park, since by its turn all the spaces in block $Y_1 = \{1,2,3\}$ are already occupied. 
    \begin{figure}[h]
        \begin{center}
        \captionsetup{justification=centering}
        \resizebox{.8\linewidth}{!}{\begin{tikzpicture}

    \draw [color = white, fill= color_gris] (0.00,0.00) ellipse (2.35 and 2.35);
    \draw [color = white, fill= white] (0.00,0.00) ellipse (1.70 and 1.70);
    \spotCircular{0.00}{0.00}{1.82}{0.4}{435.00}{465.00}{$1$}
    \spotCircular{0.00}{0.00}{1.82}{0.4}{315.00}{345.00}{$2$}
    \spotCircular{0.00}{0.00}{1.82}{0.4}{195.00}{225.00}{$3$}
    \draw [color = black, dotted,  very thick] (0.00,0.00) ellipse (0.70 and 0.70);
    \node at (0.00,0.00) {$Y_1$};
    \draw [color = white, fill= color_gris] (6.50,0.00) ellipse (2.35 and 2.35);
    \draw [color = white, fill= white] (6.50,0.00) ellipse (1.70 and 1.70);
    \spotCircular{6.50}{0.00}{1.82}{0.4}{435.00}{465.00}{$4$}
    \spotCircular{6.50}{0.00}{1.82}{0.4}{395.00}{425.00}{$5$}
    \spotCircular{6.50}{0.00}{1.82}{0.4}{355.00}{385.00}{$6$}
    \spotCircular{6.50}{0.00}{1.82}{0.4}{315.00}{345.00}{$7$}
    \spotCircular{6.50}{0.00}{1.82}{0.4}{275.00}{305.00}{$8$}
    \spotCircular{6.50}{0.00}{1.82}{0.4}{235.00}{265.00}{$9$}
    \spotCircular{6.50}{0.00}{1.82}{0.4}{195.00}{225.00}{$10$}
    \spotCircular{6.50}{0.00}{1.82}{0.4}{155.00}{185.00}{$11$}
    \spotCircular{6.50}{0.00}{1.82}{0.4}{115.00}{145.00}{$12$}
    \draw [color = black, dotted,  very thick] (6.50,0.00) ellipse (0.70 and 0.70);
    \node at (6.50,0.00) {$Y_2$};
    \draw [color = white, fill= color_gris] (13.00,0.00) ellipse (2.35 and 2.35);
    \draw [color = white, fill= white] (13.00,0.00) ellipse (1.70 and 1.70);
    \spotCircular{13.00}{0.00}{1.82}{0.4}{435.00}{465.00}{$13$}
    \spotCircular{13.00}{0.00}{1.82}{0.4}{375.00}{405.00}{$14$}
    \spotCircular{13.00}{0.00}{1.82}{0.4}{315.00}{345.00}{$15$}
    \spotCircular{13.00}{0.00}{1.82}{0.4}{255.00}{285.00}{$16$}
    \spotCircular{13.00}{0.00}{1.82}{0.4}{195.00}{225.00}{$17$}
    \spotCircular{13.00}{0.00}{1.82}{0.4}{135.00}{165.00}{$18$}
    \draw [color = black, dotted,  very thick] (13.00,0.00) ellipse (0.70 and 0.70);
    \node at (13.00,0.00) {$Y_3$};
    \carroTikZ{color_car_luky}{8.24}{2.07}{0.45}{320.00}{$1$}
    \preferenciaCircular{6.50}{0.00}{0.92}{0.6}{395.00}{425.00}{$5$}
    \carroTikZ{color_car_luky}{10.66}{-1.35}{0.45}{120.00}{$2$}
    \preferenciaCircular{13.00}{0.00}{0.92}{0.6}{195.00}{225.00}{$17$}
    \carroTikZ{color_car}{10.66}{1.35}{0.45}{60.00}{$3$}
    \preferenciaCircular{13.00}{0.00}{0.92}{0.6}{135.00}{165.00}{$17$}
    \carroTikZ{color_car_luky}{-2.34}{-1.35}{0.45}{120.00}{$4$}
    \preferenciaCircular{0.00}{0.00}{0.92}{0.6}{195.00}{225.00}{$3$}
    \carroTikZ{color_car_luky}{7.42}{-2.54}{0.45}{200.00}{$5$}
    \preferenciaCircular{6.50}{0.00}{0.92}{0.6}{275.00}{305.00}{$8$}
    \carroTikZ{color_car}{13.00}{2.70}{0.45}{360.00}{$6$}
    \preferenciaCircular{13.00}{0.00}{0.92}{0.6}{435.00}{465.00}{$18$}
    \carroTikZ{color_car}{0.00}{2.70}{0.45}{360.00}{$7$}
    \preferenciaCircular{0.00}{0.00}{0.92}{0.6}{435.00}{465.00}{$3$}
    \carroTikZ{color_car}{2.34}{-1.35}{0.45}{240.00}{$8$}
    \preferenciaCircular{0.00}{0.00}{0.92}{0.6}{315.00}{345.00}{$1$}
    \carroTikZ{color_car}{5.58}{-2.54}{0.45}{160.00}{$9$}
    \preferenciaCircular{6.50}{0.00}{0.92}{0.6}{235.00}{265.00}{$8$}
    \carroTikZ{color_car}{15.34}{1.35}{0.45}{300.00}{$10$}
    \preferenciaCircular{13.00}{0.00}{0.92}{0.6}{375.00}{405.00}{$13$}
    \carroTikZ{color_car_luky}{3.84}{0.47}{0.45}{80.00}{$11$}
    \preferenciaCircular{6.50}{0.00}{0.92}{0.6}{155.00}{185.00}{$11$}
    \carroTikZ{color_car_luky}{8.84}{-1.35}{0.45}{240.00}{$13$}
    \preferenciaCircular{6.50}{0.00}{0.92}{0.6}{315.00}{345.00}{$7$}
    \carroTikZ{color_car}{4.76}{2.07}{0.45}{40.00}{$14$}
    \preferenciaCircular{6.50}{0.00}{0.92}{0.6}{115.00}{145.00}{$11$}

    \carroTikZ{color_car}{0}{-4 + 0.5}{0.45}{0.00}{$12$}
    \draw[color = white, fill= color_gris, rounded corners=3pt] (-0.7,-4.35 + 0.5) rectangle (0.7,-5.15 + 0.5);
    \draw[fill=black, rounded corners=3pt] (-0.5,-4.5 + 0.5) rectangle (0.5,-5 + 0.5);
    \node[text = white] at (0,-4.75 + 0.5) {$\varnothing$};
    \draw[color = white, fill=color_petroleo, rounded corners=3pt] (-0.5,-4.5 - 0.25) rectangle (0.5,-5 - 0.25);
    \node[text = white] at (0,-4.75 - 0.25) {$2$};
\end{tikzpicture}}
        \end{center}
    	\caption{Example of the block parking protocol: the number on the car is its index, the number inside the black box is the parking space, and the number of the purple box is its initial preference.}
        \label{Fig. example of parking blocks}
    \end{figure}
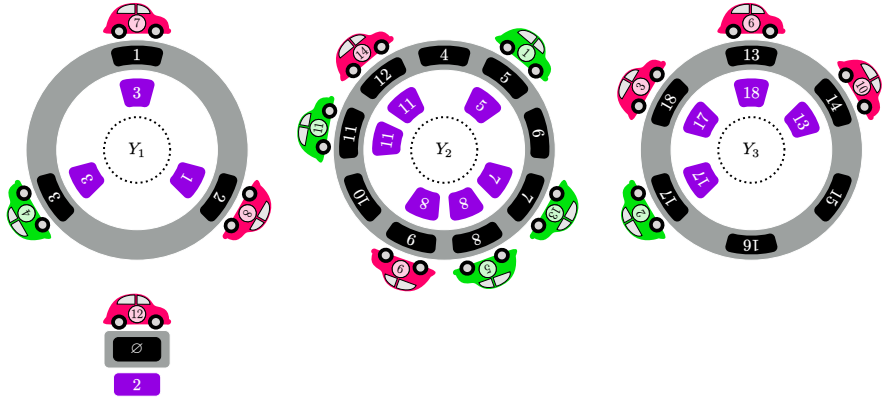
\end{example}

\begin{example}
Let $n = 3$, the number of circles, and $r = 15$ the number of cars. Consider $\yy = (4,9,5)$, so that there are $m = 4+9+5$ parking spaces. Let $l = 1$, so that the first space of the first circle is unavailable. Consider the preference word
$w = (11,3,2,10,18,2,1,16,10,14,6,7,4,6,3)\in [18]^{15}$. Figure \ref{Fig. example of parking blocks 2} illustrates the outcome of the block parking protocol for these $15$ cars; the result is 
$$\out_{\circlearrowright,1}(\yy;w) = (11, 3, 2, 10, 18, 4,\varnothing, 16, 12, 14, 6, 7,\varnothing, 8,\varnothing)$$
In this example, block $Y_1$ gets full quickly, precisely because one of its spaces is unavailable.

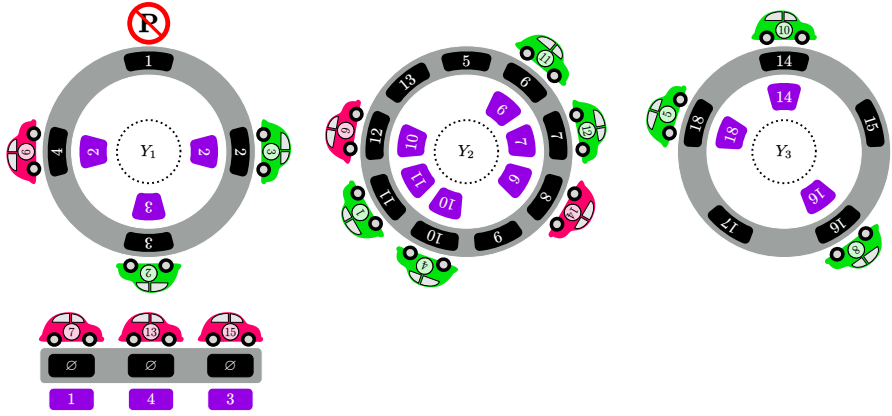
\begin{figure}[!htb]
    \centering
    \captionsetup{justification=centering}
    \resizebox{.8\linewidth}{!}{\begin{tikzpicture}

\draw [color = white, fill= color_gris] (0.00,0.00) ellipse (2.35 and 2.35);
\draw [color = white, fill= white] (0.00,0.00) ellipse (1.70 and 1.70);
\spotCircular{0.00}{0.00}{1.82}{0.4}{435.00}{465.00}{$1$}
\spotCircular{0.00}{0.00}{1.82}{0.4}{345.00}{375.00}{$2$}
\spotCircular{0.00}{0.00}{1.82}{0.4}{255.00}{285.00}{$3$}
\spotCircular{0.00}{0.00}{1.82}{0.4}{165.00}{195.00}{$4$}
\draw [color = black, dotted,  very thick] (0.00,0.00) ellipse (0.70 and 0.70);
\node at (0.00,0.00) {$Y_1$};
\draw [color = white, fill= color_gris] (7.00,0.00) ellipse (2.35 and 2.35);
\draw [color = white, fill= white] (7.00,0.00) ellipse (1.70 and 1.70);
\spotCircular{7.00}{0.00}{1.82}{0.4}{435.00}{465.00}{$5$}
\spotCircular{7.00}{0.00}{1.82}{0.4}{395.00}{425.00}{$6$}
\spotCircular{7.00}{0.00}{1.82}{0.4}{355.00}{385.00}{$7$}
\spotCircular{7.00}{0.00}{1.82}{0.4}{315.00}{345.00}{$8$}
\spotCircular{7.00}{0.00}{1.82}{0.4}{275.00}{305.00}{$9$}
\spotCircular{7.00}{0.00}{1.82}{0.4}{235.00}{265.00}{$10$}
\spotCircular{7.00}{0.00}{1.82}{0.4}{195.00}{225.00}{$11$}
\spotCircular{7.00}{0.00}{1.82}{0.4}{155.00}{185.00}{$12$}
\spotCircular{7.00}{0.00}{1.82}{0.4}{115.00}{145.00}{$13$}
\draw [color = black, dotted,  very thick] (7.00,0.00) ellipse (0.70 and 0.70);
\node at (7.00,0.00) {$Y_2$};
\draw [color = white, fill= color_gris] (14.00,0.00) ellipse (2.35 and 2.35);
\draw [color = white, fill= white] (14.00,0.00) ellipse (1.70 and 1.70);
\spotCircular{14.00}{0.00}{1.82}{0.4}{435.00}{465.00}{$14$}
\spotCircular{14.00}{0.00}{1.82}{0.4}{363.00}{393.00}{$15$}
\spotCircular{14.00}{0.00}{1.82}{0.4}{291.00}{321.00}{$16$}
\spotCircular{14.00}{0.00}{1.82}{0.4}{219.00}{249.00}{$17$}
\spotCircular{14.00}{0.00}{1.82}{0.4}{147.00}{177.00}{$18$}
\draw [color = black, dotted,  very thick] (14.00,0.00) ellipse (0.70 and 0.70);
\node at (14.00,0.00) {$Y_3$};
\carroTikZ{color_car_luky}{4.66}{-1.35}{0.45}{120.00}{$1$}
\preferenciaCircular{7.00}{0.00}{0.92}{0.6}{195.00}{225.00}{$11$}
\carroTikZ{color_car_luky}{-0.00}{-2.70}{0.45}{180.00}{$2$}
\preferenciaCircular{0.00}{0.00}{0.92}{0.6}{255.00}{285.00}{$3$}
\carroTikZ{color_car_luky}{2.70}{-0.00}{0.45}{270.00}{$3$}
\preferenciaCircular{0.00}{0.00}{0.92}{0.6}{345.00}{375.00}{$2$}
\carroTikZ{color_car_luky}{6.08}{-2.54}{0.45}{160.00}{$4$}
\preferenciaCircular{7.00}{0.00}{0.92}{0.6}{235.00}{265.00}{$10$}
\carroTikZ{color_car_luky}{11.43}{0.83}{0.45}{72.00}{$5$}
\preferenciaCircular{14.00}{0.00}{0.92}{0.6}{147.00}{177.00}{$18$}
\carroTikZ{color_car}{-2.70}{0.00}{0.45}{90.00}{$6$}
\preferenciaCircular{0.00}{0.00}{0.92}{0.6}{165.00}{195.00}{$2$}
\carroTikZ{color_car_luky}{15.59}{-2.18}{0.45}{216.00}{$8$}
\preferenciaCircular{14.00}{0.00}{0.92}{0.6}{291.00}{321.00}{$16$}
\carroTikZ{color_car}{4.34}{0.47}{0.45}{80.00}{$9$}
\preferenciaCircular{7.00}{0.00}{0.92}{0.6}{155.00}{185.00}{$10$}
\carroTikZ{color_car_luky}{14.00}{2.70}{0.45}{360.00}{$10$}
\preferenciaCircular{14.00}{0.00}{0.92}{0.6}{435.00}{465.00}{$14$}
\carroTikZ{color_car_luky}{8.74}{2.07}{0.45}{320.00}{$11$}
\preferenciaCircular{7.00}{0.00}{0.92}{0.6}{395.00}{425.00}{$6$}
\carroTikZ{color_car_luky}{9.66}{0.47}{0.45}{280.00}{$12$}
\preferenciaCircular{7.00}{0.00}{0.92}{0.6}{355.00}{385.00}{$7$}
\carroTikZ{color_car}{9.34}{-1.35}{0.45}{240.00}{$14$}
\preferenciaCircular{7.00}{0.00}{0.92}{0.6}{315.00}{345.00}{$6$}

\begin{scope}[shift={(0,2.85)}, scale=0.7 ]
  \draw[fill=none, draw=red, line width=3pt] (0,0) circle (0.6cm);
  \node[scale = 0.8] at (0,0) {\Huge \textbf{P}};
  \draw[red, line width=2pt] (-0.45,0.45) -- (0.45,-0.45);
\end{scope}

\begin{scope}[shift={(-1.7,0)}]
\draw[color = white, fill= color_gris, rounded corners=3pt] (-0.7,-4.35 + 0.5 - 0.5) rectangle (0.7 + 3.5,-5.15 + 0.5 - 0.5);

\carroTikZ{color_car}{0}{-4 + 0.5 - 0.5}{0.45}{0.00}{$7$}
\draw[fill=black, rounded corners=3pt] (-0.5,-4.5 + 0.5 - 0.5) rectangle (0.5,-5 + 0.5 - 0.5);
\node[text = white] at (0,-4.75 + 0.5 - 0.5) {$\varnothing$};
\draw[color = white, fill=color_petroleo, rounded corners=3pt] (-0.5,-4.5 - 0.25 - 0.5) rectangle (0.5,-5 - 0.25 - 0.5);
\node[text = white] at (0,-4.75 - 0.25 - 0.5) {$1$};

\carroTikZ{color_car}{0 + 1.75}{-4 + 0.5 - 0.5}{0.45}{0.00}{$13$}
\draw[fill=black, rounded corners=3pt] (-0.5 + 1.75,-4.5 + 0.5 - 0.5) rectangle (0.5 + 1.75,-5 + 0.5 - 0.5);
\node[text = white] at (0 + 1.75,-4.75 + 0.5 - 0.5) {$\varnothing$};
\draw[color = white, fill=color_petroleo, rounded corners=3pt] (-0.5 + 1.75,-4.5 - 0.25 - 0.5) rectangle (0.5 + 1.75,-5 - 0.25 - 0.5);
\node[text = white] at (0 + 1.75,-4.75 - 0.25 - 0.5) {$4$};

\carroTikZ{color_car}{0 + 3.5}{-4 + 0.5 - 0.5}{0.45}{0.00}{$15$}
\draw[fill=black, rounded corners=3pt] (-0.5 + 3.5,-4.5 + 0.5 - 0.5) rectangle (0.5 + 3.5,-5 + 0.5 - 0.5);
\node[text = white] at (0 + 3.5,-4.75 + 0.5 - 0.5) {$\varnothing$};
\draw[color = white, fill=color_petroleo, rounded corners=3pt] (-0.5 + 3.5,-4.5 - 0.25 - 0.5) rectangle (0.5 + 3.5,-5 - 0.25 - 0.5);
\node[text = white] at (0 + 3.5,-4.75 - 0.25 - 0.5) {$3$};
\end{scope}

\end{tikzpicture}}
    \caption{Example of block parking protocol with one unavailable space.}
    \label{Fig. example of parking blocks 2}
\end{figure}
\end{example}

Inspired by its counterpart notion in the classical parking function setting \cite{GesselSeo} (see also \cite{StanleyYin,harris2024parkingfunctionsfixedset,ferreri2025enumeratingvectorparkingfunctions, kenyonparkingfunctions}), we now introduce the concept of \emph{lucky cars} in the block parking protocol. A car is called lucky if it ends up occupying its initial preferred space, that is, if $\out_{\circlearrowright,l}(\yy;w)_j=w_j$. In both Figures \ref{Fig. example of parking blocks} and \ref{Fig. example of parking blocks 2}, the lucky cars can be identified as the green cars when the black rectangle (final space) and the purple rectangle (initial preference) are equal.

\begin{definition}
    With the above notation for the block parking protocol. The \emph{set of lucky cars} of $w$ is defined as
    \[\Lucky_{\circlearrowright,l}(\yy; w) := \Big\{ j \in [r] : \out_{\circlearrowright,l}(\yy;w)_j = w_j \Big\},\]
    and the \emph{number of lucky cars} of $w$ is denoted as $\lucky_{\circlearrowright,l}(\yy; w) := \#\Lucky_{\circlearrowright,l}(\yy; w)$.
\end{definition}

Recall that if we fix a multipermutation $\pi \in \mulPerm{\ss}$, we can extend it to a $\yy$-extended permutation $\sigma \in \ExtPerm{\ss}{\yy}$ such that $\sigma_j \in Y_{\pi_j}$. So, we can define the following quantities:

\begin{definition}
Let $\pi \in \mulPerm{\ss}$ and let $L \subseteq [r]$. We define:
\begin{align*}
\lucky^{(\alpha)}_{\circlearrowright,l}(\yy; \pi, L) &:= \#\{ \sigma \in \ExtPerm{\ss}{\yy} :  \pi \text{ is extended to } \sigma \text{ and } \Lucky_{\circlearrowright,l}(\yy; \sigma) \subseteq L \},\\
\lucky^{(\beta)}_{\circlearrowright,l}(\yy; \pi, L) &:= \#\{ \sigma \in \ExtPerm{\ss}{\yy} : \pi \text{ is extended to } \sigma \text{ and } \Lucky_{\circlearrowright,l}(\yy; \sigma) = L \}.
\end{align*}
\end{definition}

\begin{remark}\label{remk. blucky inclusion exclusion}
By inclusion-exclusion, the two quantities are related by
\begin{align*}
\lucky^{(\alpha)}_{\circlearrowright,l}(\yy; \pi, L) &= \sum_{S \subseteq L} \lucky^{(\beta)}_{\circlearrowright,l}(\yy; \pi, S),\\
\lucky^{(\beta)}_{\circlearrowright,l}(\yy; \pi, L) &= \sum_{S \subseteq L} (-1)^{|L|-|S|}\lucky^{(\alpha)}_{\circlearrowright,l}(\yy; \pi, S).
\end{align*}
\end{remark}

To analyze these quantities, recall the definition of the $j$-th partial multiplicity of $\pi$ (see Definition~\ref{def. algebraic contributions}),
\[\mu(\pi)_j := \#\{ k \le j : \pi_k = \pi_j \},\]
which counts how many cars up to time $j$ are assigned to the same block as the $j$-th car.
If $\sigma_j \in Y_i$, then $\mu(\pi)_j-1$ is precisely the number of cars that have already attempted to park in block $Y_i$ before the $j$-th car.

When $l>0$, the first space of block $Y_l$ is unavailable. This is encoded by the correction term $\delta_{\pi_j,l}$, so that
\[o_{j,l}(\pi) := \mu(\pi)_j - 1 + \delta_{\pi_j,l}\]
represents the total number of spaces in $Y_{\pi_j}$ that are unavailable to the $j$-th car at the moment it arrives, including the ones occupied and those unavailable.

Since cars search for parking spaces cyclically within their block, the $j$-th car finds a space to park if and only if
\[o_{j,l}(\pi) + 1 \le y_{\pi_j}.\]
This leads to a split of the set $[r]$ by:
\[A_{\pi,l} := \{ j \in [r] : o_{j,l}(\pi) + 1 \le y_{\pi_j} \}, \qquad
I_{\pi,l} := \{ j \in [r] : o_{j,l}(\pi) + 1 > y_{\pi_j} \} = [r]\setminus A_{\pi,l}.\]
The set $A_{\pi,l}$ consists of cars that always succeed in parking, regardless of their precise initial preference within the block. Meanwhile, $I_{\pi,l}$ collects the indices of cars that necessarily fail to park, since their block is already saturated when they arrive.

\begin{proposition}\label{prop. ecuaciones de blucky by pi}
Let $\ss$ and $\yy$ be vectors of non-negative integers, let $\pi \in \mulPerm{\ss}$ and $L \subseteq A_{\pi,l}$. Then:
\begin{equation}\label{eq. blucky alpha}
    \lucky^{(\alpha)}_{\circlearrowright,l}(\yy; \pi, L) = 
    \prod_{j \in L \sqcup I_{\pi,l}} y_{\pi_j} \cdot 
    \prod_{j \in (A_{\pi,l} \setminus L)}o_{j,l}(\pi),
\end{equation}
\begin{equation}\label{eq. blucky beta}
    \lucky^{(\beta)}_{\circlearrowright,l}(\yy; \pi, L) = 
    \prod_{j \in L} \big(y_{\pi_j} - o_{j,l}(\pi)\big) \cdot
    \prod_{j \in (A_{\pi,l} \setminus L)}o_{j,l}(\pi) \cdot
    \prod_{j \in I_{\pi,l}} y_{\pi_j}.
\end{equation}

In particular, if $L \cap I_{\pi,l} \neq \varnothing$, then $\lucky^{(\alpha)}_{\circlearrowright,l}(\yy; \pi, L) = \lucky^{(\beta)}_{\circlearrowright,l}(\yy; \pi, L) = 0.$
\end{proposition}

\begin{proof}
The argument is similar in both cases. For $\sigma \in \ExtPerm{\ss}{\yy}$ with base $\pi$, each position $j \in [r]$ satisfies $\sigma_j \in Y_{\pi_j}$. We will divide the proof into three cases, depending on where $j$ is.
\begin{enumerate}[(I)]
    \item If $j \in L \subseteq A_{\pi,l}$, then the $j$-th car may or may not be lucky, giving $y_{\pi_j}$ possible choices for $\sigma_j$. If it is required to be lucky, it cannot choose any of the $o_{j,l}(\pi)$ previously occupied spaces in $Y_{\pi_j}$, so there are $y_{\pi_j} - o_{j,l}(\pi)$ valid choices.
    \item If $j \in A_{\pi,l} \setminus L$, then the $j$-th car cannot be lucky, and thus must select one of the $o_{j,l}(\pi)$ previously occupied spaces.
    \item If $j \in I_{\pi,l}$, then the block $Y_{\pi_j}$ is already full. The car cannot be lucky and may take any of the $y_{\pi_j}$ possible spaces.
\end{enumerate}

The condition $L \subseteq A_{\pi,l}$ is necessary because if there is a car in $I_{\pi,l}$, it can never be lucky. Thus, if $L \cap I_{\pi,l} \neq \varnothing$, the quantities vanish.
\end{proof}

The formulas obtained in Proposition \ref{prop. ecuaciones de blucky by pi} and the algebraic contributions $\alg_{l}^{(\ \cdot \ )}(\yy;\pi,L)$ introduced in Definition \ref{def. algebraic contributions} bear evident similarities. However, there is an essential difference: the algebraic contributions are defined for every subset $L\subseteq[r]$, whereas the lucky quantities $\blucky_{(\ \cdot \ ),l}(\pi, \yy, L)$ depend on the set $I_{\pi,l}$ of cars that can never be lucky. Nonetheless, there is a general relationship between these two families of quantities.

\begin{proposition}\label{prop. blucky-vs-alg}
Let $\ss = (s_1,\dots,s_n)$ and $\yy = (y_1,\dots,y_n)$ be a vector of non-negative integers. For all $L \subseteq [r]$, $r = \sum_{i=1}^{n}s_i$, and for all $\pi \in \mulPerm{\ss}$, we have that:
\begin{enumerate}[\normalfont (I)]
    \item $\displaystyle\lucky^{(\alpha)}_{\circlearrowright,l}(\yy; \pi, L) \;=\; \sum_{S \subseteq I_{\pi,l}}(-1)^{|S\cap L|}\,\alg_{l}^{(\alpha)}(\yy;\pi,L \cup S)$ \\
    \item $\displaystyle\lucky^{(\beta)}_{\circlearrowright,l}(\yy; \pi, L) \;=\; \sum_{S \subseteq I_{\pi,l}}(-1)^{|S\cap L|}\,\alg_{l}^{(\beta)}(\yy;\pi,L \cup S)$
\end{enumerate}
In particular, if $L\subseteq A_{\pi,l}$ then:
$$\lucky^{(\alpha)}_{\circlearrowright,l}(\yy; \pi, L) \;=\; \sum_{S \subseteq I_{\pi,l}}\alg_{l}^{(\alpha)}(\yy;\pi,L \cup S),
\qquad
\lucky^{(\beta)}_{\circlearrowright,l}(\yy; \pi, L) \;=\; \sum_{S \subseteq I_{\pi,l}}\alg_{l}^{(\beta)}(\yy;\pi,L \cup S).$$
\end{proposition}

\begin{proof}
We prove the identity for the $(\beta)$ version, since the relationship between $(\alpha)$ and $(\beta)$ versions are determined by an inclusion-exclusion argument by Lemma~\ref{lemma. relation into beta and alpha alg} and Remark~\ref{remk. blucky inclusion exclusion}. 

The proof is divided into two cases depending on whether $L \cap I_{\pi,l}$ is empty or not.
\begin{enumerate}
    \item[\textbf{Case 1.}] $L \cap I_{\pi,l} \not = \emptyset$: In this case we know that $\lucky^{(\beta)}_{\circlearrowright,l}(\yy; \pi, L) = 0$, any index in $I_{\pi,l}$ corresponds to a car that arrives after its block is already saturated and therefore cannot be lucky.
    
    Given that $L \cap I_{\pi,l} \not = \emptyset$, then there is an element $j_0 \in L \cap I_{\pi,l}$. Fix this element, and consider the map $\vartheta_{j_0} : \mathcal{P}(I_{\pi,l}) \longrightarrow \mathcal{P}(I_{\pi,l})$, defined as:
    \[\vartheta_{j_0}(S) = \begin{cases}
    S \cup \{j_0\} &\text{if }j_0 \not\in S\\
    S \setminus \{j_0\} &\text{if }j_0 \in S\end{cases} \ \ .\]
    Note that $j_0 \in I_{\pi,l}$, the map $\vartheta_{j_0}$ is an involution in $\mathcal{P}(I_{\pi,l})$ without fixed points. \\
    
    Since $\alg_{l}^{(\beta)}(\yy;\pi, L \cup S) = \alg_{l}^{(\beta)}(\yy;\pi, L \cup \vartheta_{j_0}(S))$, because $L \cup S = L \cup \vartheta_{j_0}(S)$, and $(-1)^{|S\cap L|} = - (-1)^{|\vartheta_{j_0}(S)\cap L|}$, then $\vartheta_{j_0}$ is a sign-reversing involution without fixed points, therefore the sum vanishes.
    \item [\textbf{Case 2.}] $L \cap I_{\pi,l}  = \emptyset$: Therefore for all $S\subseteq I_{\pi,l}$ we have that $|S \cap L| = 0$, then we will show that: 
    \[\lucky^{(\beta)}_{\circlearrowright,l}(\yy; \pi, L) \;=\; \sum_{S \subseteq I_{\pi,l}}\alg_{l}^{(\beta)}(\yy;\pi,L \cup S).\]
    If we expand the sum by the definition of product of $\alg_{l}^{(\beta)}(\yy ; \pi, S)$ since for all $S \cap L = \emptyset$ we can factor when $j \in L$, then we have that:
    \begin{multline*}
        \sum_{S\subseteq I_{\pi,l}}\alg_{l}^{(\beta)}(\yy ; \pi,L\cup S) =\\
        \prod_{j\in L}\big(y_{\pi_j}-o_{j,l}(\pi)\big)\cdot\prod_{j\in A_{\pi,l}\setminus L}o_{j,l}(\pi)
        \quad\cdot
        \sum_{S\subseteq I_{\pi,l}}\left(\prod_{j\in S}\big(y_{\pi_j}-o_{j,l}(\pi)\big)\prod_{j\in I_{\pi,l}\setminus S}o_{j,l}(\pi)\right).
    \end{multline*}
    Since $(y_{\pi_j}-o_{j,l}(\pi)) + o_{j,l}(\pi)=y_{\pi_j}$ for every $j$, then when we expand the sum over $S \subseteq I_{\pi,l}$, yields 
    \begin{equation*}
        \sum_{S\subseteq I_{\pi,l}}\alg_{l}^{(\beta)}(\yy ; \pi,L\cup S) = \prod_{j \in L} \big(y_{\pi_j} - o_{j,l}(\pi)\big) \cdot
    \prod_{j \in (A_{\pi,l} \setminus L)}o_{j,l}(\pi) \cdot
    \prod_{j \in I_{\pi,l}} y_{\pi_j},
    \end{equation*}
    which is exactly the formula for $\lucky^{(\beta)}_{\circlearrowright,l}(\yy; \pi, L)$ in Proposition~\ref{prop. ecuaciones de blucky by pi}.\qedhere
\end{enumerate}
\end{proof}

As a consequence of the preceding corollary, we obtain the following specialization, which will be instrumental to deduce Step (i) in the scheme of the main proof.

\begin{corollary}\label{cor. full-capacity}
If $y_i\ge s_i + \delta_{i,l}$ for every $i\in[n]$ then $I_{\pi,l}=\varnothing$ for every $\pi \in \mulPerm{\ss}$. Consequently, for every $\pi$ and every $L\subseteq[r]$ we have
\[\lucky^{(\beta)}_{\circlearrowright,l}(\yy; \pi, L)=\alg_{l}^{(\beta)}(\yy ; \pi,L).\]
In particular, by summing over all multipermutations $\pi \in \mulPerm{\ss}$ and all $L \subseteq [r]$, then by Proposition~\ref{prop. refine of f_s,y by pi and L} yields
\[{f}_{\ss,l}(t,\yy)=\sum_{L\subseteq[r]}\sum_{\pi\in\mulPerm{\ss}}\lucky^{(\beta)}_{\circlearrowright,l}(\yy; \pi, L)\,t^{|L|}
=\sum_{\sigma\in\ExtPerm{\ss}{\yy}} t^{\lucky_{\circlearrowright,l}(\yy; \sigma)}.\]
\end{corollary}

We are ready to prove a combinatorial interpretation for the coefficients of the polynomials appearing in equation~\eqref{eq. polynomial of algebraic lucky in}.

\begin{theorem}\label{thm. coefficients-count}
Let $\ss,\yy \in \Z_{\geq 0}^{n}$ and let $r=\sum_i s_i$, and assume that $y_i\ge s_i + \delta_{i,l}$ for every $i$. Consider the polynomial
\[{g}_{\ss,l}(t,\yy) := \prod_{i=1}^n\bbinom{y_i t + \delta_{i,l}}{s_i}=\sum_{j=0}^r c_j\,t^j(1+t)^{r-j},\]
so that ${f}_{\ss,l}(t,\yy) = r!\mathscr{M}({g}_{\ss,l}(t,\yy))$. Then for each $j$ the integer $r!\,c_j$ equals the number of $\yy$-extended permutations $\sigma\in\ExtPerm{\ss}{\yy}$ with exactly $j$ lucky cars under the block parking protocol.
\[r!\,c_j \;=\; \#\{\sigma\in\ExtPerm{\ss}{\yy}:\lucky_{\circlearrowright,l}(\yy; \sigma)=j\}.\]
\end{theorem}

\begin{proof}
If $y_i\ge s_i + \delta_{i,l}$ Corollary~\ref{cor. full-capacity} and by Lemma~\ref{lem. trasformacion polynomiasl} we obtain that
\[{f}_{\ss,l}(t,\yy)=\sum_{\sigma\in\ExtPerm{\ss}{\yy}} t^{\lucky_{\circlearrowright,l}(\yy; \sigma)} = r!\cdot \mathscr{M}({g}_{\ss,l}(\cdot,\yy))(t).\]
Writing $f_{\ss,\yy}(t)=\sum_{j=0}^r a_j t^j$, then we have that $c_j={a_j}/{r!}$. But $a_j$ is precisely the number of $\yy$-extended permutations with exactly $j$ lucky cars, so $r!c_j$ equals the desired number.
\end{proof}

\subsection{Applying the machinery to generalized permutohedra} We now explain how the combinatorial machinery developed in the previous section fits into the general strategy towards the proof of Theorem~\ref{thm:main-intro}, and how it applies to $\yy$-generalized permutohedra.

\begin{remark}\label{rk. interpretacion subset of com}
Both the set of $H$-Draconian sequences and the family $\mathcal{I}_n$ are particular subsets of the set of weak compositions.
More generally, the constructions of this section apply to any subset $\mathcal{S}$ of weak compositions of $r$ in $n$ parts for which the associated family
$$\ExtPerm{\mathcal{S}}{y}:=\bigsqcup_{\ss\in\mathcal{S}}\ExtPerm{\ss}{\yy}$$
and the polynomial
$$g_{\mathcal{S},l}(t,\yy) :=\sum_{\ss \in \mathcal{S}}\prod_{i=1}^n\bbinom{y_i t + \delta_{i,l}}{s_i}$$
are well defined.
When $y_i \geq r + \delta_{i,l}$, the previous arguments show that the coefficients of $\mathscr{M}(g_{\mathcal{S},l}(t,\yy))$ admit a combinatorial interpretation in terms of lucky cars in $\ExtPerm{\mathcal{S}}{y}$.
In the present work, however, we restrict our attention to those subsets arising naturally from $\yy$-generalized permutohedra, and in particular from Pitman--Stanley polytopes.
\end{remark}

\begin{theorem}\label{thm:step-i-magic}
Let $\yy = (y_1, \dots, y_m)$ be a vector of non-negative integers such that $y_i \geq n - 1$ for all $i$ and $y_1 \geq n$.
Let $\yPerm_H({\bf y})$ be a $\mathbf{y}$-generalized permutohedron associated to a bipartite graph $H\subset K_{m,n}$.
Then the Ehrhart polynomial $\ehr_{\yPerm_H({\bf y})}(t)$ is magic positive.
Moreover, writing
\[\mathfrak{D}_{\yy}(H) := \bigsqcup_{{\bf a}\in\Draco(H)}\ExtPerm{\bf a}{\yy},\]
the coefficient of $t^j(1+t)^{n-1-j}$ equals, up to the normalization factor $(n-1)!$, the number of words in $\mathfrak{D}_{\yy}(H)$ with exactly $j$ lucky cars under the block parking protocol without the first available space.
\end{theorem}

This establishes Step (i) of the strategy described above.
Since Pitman--Stanley polytopes form a particular family of $\yy$-generalized permutohedra, the same interpretation applies to $\ehr_{\Pi_n(\yy)}(t)$, yielding Step (ii) through the next corollary. We note that recently Deligeorgaki, Han, and Solus \cite[Theorem~20]{DHS} also proved the magic positivity without providing a combinatorial interpretation for these coefficients. The next result has such a combinatorial interpretation in this case.

\begin{corollary}\label{cor:step-2-magic}
Let $\yy = (y_1, \dots, y_n)$ be a vector of non-negative integers such that $y_i \geq n - 1$ for all $i$ and $y_1 \geq n$.
Then the coefficients of the Ehrhart polynomial $\ehr_{\Pi_n(\yy)}(t)$ in the magic basis admit a combinatorial interpretation:
up to the normalization factor $(n-1)!$, the coefficient of $t^j(1+t)^{n-1-j}$ equals the number of $\yy$-parking functions with exactly $j$ lucky cars under the block parking protocol without the first available space.
\end{corollary}

It is crucial to emphasize that the arguments employed to deduce the last theorem and its corollary rely \emph{in an essential way} on the assumption that the entries of $\yy$ are sufficiently large, which ensures that all cars can potentially park.
When this condition fails, the block parking protocol must be modified. In the next section, we show that for Pitman--Stanley polytopes such modifications lead to a valid combinatorial interpretation for all $\yy\in\mathbb{Z}_{>0}^n$, which will thus complete Step~(iii) and hence the proof of Theorem~\ref{thm:main-intro}.

\section{Magic positivity and lucky statistic on \texorpdfstring{$\yy$}--Parking Functions}\label{sec:magic}

In this section we finalize the arguments of the proof of Theorem~\ref{thm:main-intro}. In Section~\ref{ssec: gen parking} we introduce the generalized parking functions and relate them to ${\bf y}$-parking functions. In Section~\ref{protocol gen parking functions} we give the parking protocol for words in a fixed number of spaces and the lucky statistic and derive key properties of these. In Section~\ref{ssec: lucky subset lucky cars} we use this statistic to give the interpretation of the coefficients of the magic basis of the Ehrhart polynomial of the Pitman--Stanley polytope. This interpretation is proved in Section~\ref{ssec: algo preserving pf}.

\subsection{Generalized Parking Functions} \label{ssec: gen parking}

As recapitulated in Section~\ref{subsec:block-parking-protocol}, the classical notion of parking functions refers to the set of words $\Pf_n \subseteq [n]^n$ for which all cars successfully park. Among the several generalizations of this notion proposed in the literature, one particularly useful for our endeavors is the one studied by Kenyon and Yin in \cite{kenyonparkingfunctions}. We say that a word $w \in [m]^n$ is a \textit{generalized parking function} if all $n$ cars successfully park in the $m$ available spaces. The set of generalized parking functions is denoted by $\Pf(n,m)$. Kenyon and Yin \cite{kenyonparkingfunctions} provide the following characterization that allows one to determine whether a word belongs to $\Pf(n,m)$.

\begin{proposition}[{\cite[Prop.~2.1]{kenyonparkingfunctions}}]\label{Pro. all successfully}
Let $w \in [m]^n$ be a word whose non-decreasing rearrangement is $b_1 \leq b_2 \leq \cdots \leq b_n$. Then $w \in \Pf(n,m)$ if and only if
\[b_j \leq m - n + j \quad \text{for all } j = 1, \dots, n.\]
\end{proposition}

The above proposition provides a useful criterion for determining whether a preference word $w \in [m]^n$ allows all $n$ cars to park successfully.

\begin{remark}
    By Proposition~\ref{Pro. all successfully} and Definition~\ref{def:y-parking-functions} of $\yy$-parking functions there is a connection between generalized parking functions and $\yy$-parking functions. Specifically, by setting $\yy = (m - n + 1, 1, 1, \dots, 1)$, we obtain the identity $\Pf(\yy) = \Pf(n,m)$. 
\end{remark}

From this point on, whenever the interpretation of the parking-protocol is used, we adopt the convention that $n$ denotes the number of cars and $m$ the number of parking spaces.

\begin{proposition}\label{Pro words on Dyck path successfully park}
    Let $\yy=(y_1,\dots,y_k)\in \mathbb{Z}_{> 0}$. For all $\yy$-parking functions $\sigma$, if $\sigma_j$ is the preference space for the $j$-th car in $m = \sum_{i=1}^{k}y_i$ available spaces, then all $n$ cars successfully park, and therefore $\Pf_n(\yy) \subseteq \Pf(n,m)$.
\end{proposition}

\begin{proof}
    By Proposition \ref{Pro. all successfully}, it suffices to show that for any $\yy$-parking function $\sigma \in \Pf_n(\yy)$ whose non-decreasing rearrangement $b_1 \leq b_2 \leq \dots \leq b_n$ satisfy $b_i \leq m - n + i$.   By Definition~\ref{def:y-parking-functions},
    \[b_i \leq y_1 + \cdots + y_i.\]
    In addition, we have that each $y_i \geq 1$, therefore:
    \[y_1 + \cdots + y_i = m - (y_{i+1} + \dots + y_n) \leq m - n + i.\]
    This, combined with Proposition~\ref{Pro. all successfully}, completes the proof.
\end{proof}

\subsection{Lucky statistic on \texorpdfstring{$\yy$}--Parking Functions} \label{protocol gen parking functions}

We consider a generalization of the classical parking protocol, formulated in terms of words. 

\begin{protocol*}{Generalized classical parking protocol in $m$ spaces}
    Fix $l \in \{0,1\}$ and let $w \in [m]^r$ be a word of length $r$ over an alphabet of $m$ letters. We interpret $w = (w_1, \dots, w_r)$ as the parking preferences of $r$ cars that try to park in $m$ available spaces.
    \begin{itemize}
        \item If $l = 0$ all spaces are available,
        \item if $l = 1$ the first space is unavailable.
    \end{itemize}
    Each car $j \in [r]$ tries to park in its preferred space $w_j$. If that space is already occupied or is not available, the car moves to the next available space on the right. If no available space is found before reaching the end, the car is considered to have failed to park. 
\end{protocol*}

We denote by $\out_{m,l}(w)$ the output vector of the process, where $\out_{m,l}(w)_j$ represents the space where the $j$-th car parks, or $\varnothing$ if it fails to park.

We define the following associated statistics
\begin{itemize}
    \item The lucky set: $\Lucky_{m,l}(w) = \{ j \in [r] \mid \out_{m,l}(w)_j = w_j \}$.
    \item The number of lucky cars: $\lucky_{m,l}(w) = \#\Lucky_{m,l}(w)$.
\end{itemize}

Example \ref{Ex: pakring wituh first} shows the output in the special version of the parking protocol where the parameter $l=1$.

\begin{example}\label{Ex: pakring wituh first}
Consider the word $w = (6,3,1,6,2,7,6)$ with $m=8$ and $r = 7$, then the output under the parking protocol in $m$ spaces and $l=1$, then: 
$$\out_{m,1}(w) = (6,3,2,7,4,8,\varnothing)$$
See figure \ref{Fig. example of parking 3} to illustrate this process.
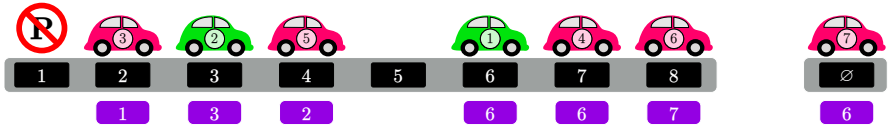
\begin{figure}[!htb]
    \centering
    \captionsetup{justification=centering}
    \resizebox{.8\linewidth}{!}{\begin{tikzpicture}

\draw [color = white, fill= color_gris, rounded corners=3pt] (-0.7,-0.35) rectangle (12.5,-1);
\draw [color = white, fill= color_gris, rounded corners=3pt] (14.1,-0.35) rectangle (15.65,-1);
\draw[fill=black, rounded corners=1pt] (-0.5,-0.45) rectangle (0.5,-0.9);
\draw[fill=black, rounded corners=1pt] (-0.5 + 1.5,-0.45) rectangle (0.5 + 1.5,-0.9);
\draw[fill=black, rounded corners=1pt] (-0.5 + 1.5*2 + 0.2*1,-0.45) rectangle (0.5 + 1.5*2 + 0.2*1,-0.9);
\draw[fill=black, rounded corners=1pt] (-0.5 + 1.5*3 + 0.2*2,-0.45) rectangle (0.5 + 1.5*3 + 0.2*2,-0.9);
\draw[fill=black, rounded corners=1pt] (-0.5 + 1.5*4 + 0.2*3,-0.45) rectangle (0.5 + 1.5*4 + 0.2*3,-0.9);
\draw[fill=black, rounded corners=1pt] (-0.5 + 1.5*5 + 0.2*4,-0.45) rectangle (0.5 + 1.5*5 + 0.2*4,-0.9);
\draw[fill=black, rounded corners=1pt] (-0.5 + 1.5*6 + 0.2*5,-0.45) rectangle (0.5 + 1.5*6 + 0.2*5,-0.9);
\draw[fill=black, rounded corners=1pt] (-0.5 + 1.5*7 + 0.2*6,-0.45) rectangle (0.5 + 1.5*7 + 0.2*6,-0.9);
\draw[fill=black, rounded corners=1pt] (-0.5 + 1.5*9 + 0.2*7,-0.45) rectangle (0.5 + 1.5*9 + 0.2*7,-0.9);

\node[text = white] at (0,-0.7) {$1$};
\node[text = white] at (1.5,-0.7) {$2$};
\node[text = white] at (3 + 0.2*1,-0.7) {$3$};
\node[text = white] at (4.5 + 0.2*2,-0.7) {$4$};
\node[text = white] at (6 + 0.2*3,-0.7) {$5$};
\node[text = white] at (7.5 + 0.2*4,-0.7) {$6$};
\node[text = white] at (9 + 0.2*5,-0.7) {$7$};
\node[text = white] at (10.5 + 0.2*6,-0.7) {$8$};
\node[text = white] at (13.5 + 0.2*7,-0.7) {$\varnothing$};

\begin{scope}[shift={(0,0.2)}, scale=0.7 ]
  \draw[fill=none, draw=red, line width=3pt] (0,0) circle (0.6cm);
  \node[scale = 0.8] at (0,0) {\Huge \textbf{P}};
  \draw[red, line width=3pt] (-0.45,0.45) -- (0.45,-0.45);
\end{scope}

\carroTikZ{color_car}{1.5}{0}{0.45}{0.00}{$3$}
\carroTikZ{color_car_luky}{1.5*2 + 0.2}{0}{0.45}{0.00}{$2$}
\carroTikZ{color_car}{1.5*3 + 0.2*2}{0}{0.45}{0.00}{$5$}
\carroTikZ{color_car_luky}{1.5*5 + 0.2*4}{0}{0.45}{0.00}{$1$}
\carroTikZ{color_car}{1.5*6 + 0.2*5}{0}{0.45}{0.00}{$4$}
\carroTikZ{color_car}{1.5*7 + 0.2*6}{0}{0.45}{0.00}{$6$}
\carroTikZ{color_car}{1.5*9 + 0.2*7}{0}{0.45}{0.00}{$7$}

\draw[color = white, fill=color_petroleo, rounded corners=3pt] (-0.5 + 1.5,-0.45 - 0.7) rectangle (0.5 + 1.5,-0.9 - 0.7);
\draw[color = white, fill=color_petroleo, rounded corners=3pt] (-0.5 + 1.5*2 + 0.2*1,-0.45 - 0.7) rectangle (0.5 + 1.5*2 + 0.2*1,-0.9 - 0.7);
\draw[color = white, fill=color_petroleo, rounded corners=3pt] (-0.5 + 1.5*3 + 0.2*2,-0.45 - 0.7) rectangle (0.5 + 1.5*3 + 0.2*2,-0.9 - 0.7);
\draw[color = white, fill=color_petroleo, rounded corners=3pt] (-0.5 + 1.5*5 + 0.2*4,-0.45 - 0.7) rectangle (0.5 + 1.5*5 + 0.2*4,-0.9 - 0.7);
\draw[color = white, fill=color_petroleo, rounded corners=3pt] (-0.5 + 1.5*6 + 0.2*5,-0.45 - 0.7) rectangle (0.5 + 1.5*6 + 0.2*5,-0.9 - 0.7);
\draw[color = white, fill=color_petroleo, rounded corners=3pt] (-0.5 + 1.5*7 + 0.2*6,-0.45 - 0.7) rectangle (0.5 + 1.5*7 + 0.2*6,-0.9 - 0.7);
\draw[color = white, fill=color_petroleo, rounded corners=3pt] (-0.5 + 1.5*9 + 0.2*7,-0.45 - 0.7) rectangle (0.5 + 1.5*9 + 0.2*7,-0.9 - 0.7);

\node[text = white] at (1.5,-0.7 - 0.7) {$1$};
\node[text = white] at (3 + 0.2*1,-0.7 - 0.7) {$3$};
\node[text = white] at (4.5 + 0.2*2,-0.7 - 0.7) {$2$};
\node[text = white] at (7.5 + 0.2*4,-0.7 - 0.7) {$6$};
\node[text = white] at (9 + 0.2*5,-0.7 - 0.7) {$6$};
\node[text = white] at (10.5 + 0.2*6,-0.7 - 0.7) {$7$};
\node[text = white] at (13.5 + 0.2*7,-0.7 - 0.7) {$6$};

\end{tikzpicture}}
    \caption{Example of parking protocol with the first space unavailable.}
    \label{Fig. example of parking 3}
\end{figure}
\end{example}

In the previous section, we established that for any subset $\mathcal{S}$ of weak compositions of $r$ into $n$ parts, we could associate a collection of $\yy$-extended permutations:
$$\ExtPerm{\mathcal{S}}{y}:=\bigsqcup_{\ss\in\mathcal{S}}\ExtPerm{\ss}{\yy}$$
and the polynomial
$$g_{\mathcal{S},l}(t,\yy) :=\sum_{\ss \in \mathcal{S}}\prod_{i=1}^n\bbinom{y_i t + \delta_{i,l}}{s_i}$$
where $l = 0,\ldots,n$ is a fixed parameter.
If $y_i \geq r + \delta_{i,l}$ for each $i$, then the magic coefficients of $g_{\mathcal{S},l}(t,\yy)$ admit a combinatorial interpretation via the block parking protocol. This interpretation relies on the assumption that $\yy$ is sufficiently large.

In this section we will show that the hypothesis that $y_i \geq 1$ is sufficient for the polynomial associated with the set $\mathcal{I}_n$ to admit a combinatorial interpretation over the set of $\yy$-parking functions via the parking protocol with the parameters $l = 0,1$.

\begin{theorem}\label{The. magic for a y-parking functions}
     Let $\mathbf{y} = (y_1, \dots, y_k) \in \mathbb{Z}_{>0}^k$. Consider the polynomials
     \begin{equation}\label{eq. g polynomial of y-parking}
         g_{\Pf_n(\yy),0}(t) := \sum_{\ss \in \mathcal{I}_{n}}\prod_{i=1}^n\bbinom{y_i t}{s_i} \quad \text{ and } \quad g_{\Pf_n(\yy),1}(t) :=\sum_{\ss \in \mathcal{I}_{n}}\bbinom{y_1 t + 1}{s_1}\prod_{i=2}^n\bbinom{y_i t}{s_i}.
     \end{equation}
     Fix $l\in \{0,1\}$, and call $c_j$ the coefficient of degree $j$ of $\mathscr{M}(g_{\Pf_n(\yy),l})$. The quantity $d_j := n!\cdot c_j$ counts the number of $\yy$-parking functions with exactly $j$ lucky cars under our parking protocol. In particular, $c_j\geq 0$.
\end{theorem}

This result clearly implies Theorem~\ref{thm:main-intro} since $\ehr_{\Pi_n(\yy)}(t) = g_{\Pf_n(\yy),1}(t)$, and is the main goal for the remainder of this article.  

\begin{example} \label{ex:lucky PS y=(21)}
For $\yy=(2,1)$ and the $\yy$-parking functions of Ex.~\ref{ex: y-parking}, the parking protocol with $l = 1$  in Figure~\ref{Fig. example of parking 4 magic} gives $d_0=2$, $d_1=4$ and $d_2=2$. Indeed, in agreement with Ex.~\ref{ex:Ehr PS2},
\[
2\cdot \ehr_{\Pi_2(2,1)}(t) = 2\cdot g_{\Pf_2(2,1),1}(t) = 2t^0(1+t)^2 + 4t^1(1+t)^1+2t^2(1+t)^0.
\]

\begin{figure}
    \centering
    \captionsetup{justification=centering}
    \resizebox{\linewidth}{!}{\input{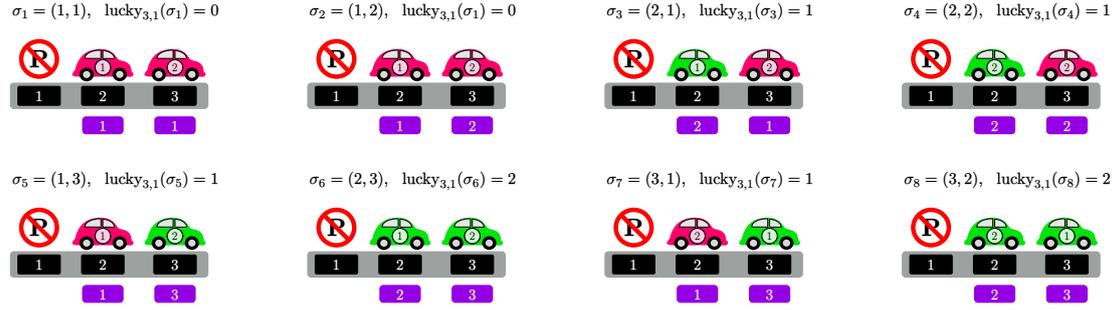}}
    \caption{The parking protocol is illustrated for each $\yy$-parking function $\sigma_i \in \Pf_2(2,1)$, where the green cars are lucky cars.}
    \label{Fig. example of parking 4 magic}
\end{figure}
\end{example}

As in the previous section, such an interpretation requires a precise control of the parking protocol: we must understand under which conditions a preference word allows all cars to park successfully.

For $\mathbf{y}$-parking functions (and, more generally, for generalized parking functions), Proposition~\ref{Pro words on Dyck path successfully park} ensures all cars park when $l=0$. In order to treat the case $l=1$, we need a slightly weaker but still sufficient fact: under this modified protocol, a word  $w$  in $\Pf(n,m)$  allows all cars to park except possibly the last one. The following lemma collects the properties that will be useful in proving this fact.

\begin{lemma}\label{lemma. properties parking generalized}
    Let $w = (w_1,\dots,w_n) \in \Pf(n,m)$ be a generalized parking function, and let $w' = (w_{i_1},\dots,w_{i_n})$ be a permutation of its entries. Then the following properties hold for the parking protocol with parameter $l = 1$:
    \begin{enumerate}[\normalfont(i)]
        \item \label{it:lemma. properties parking generalized I} $w'$ is also a generalized parking function.
        
        \item \label{it:lemma. properties parking generalized II} All cars with preference word $w$ park successfully if and only if all cars with preference word $w'$ park successfully.
        
        \item \label{it:lemma. properties parking generalized III} If a prefix $(w_1,\dots,w_j)$ contains a car that fails to park, then any permutation $(w_{i_1},\dots,w_{i_j})$ of this prefix also contains a car that fails to park.
        
        \item \label{it:lemma. properties parking generalized IV} Suppose that when $l=0$, all cars with preferences in the prefix $(w_1,\dots,w_j)$ park successfully, but when $l=1$ some car fails to park. Then there exists an index $i \leq j$ such that $w_i = 1$.
    \end{enumerate}
\end{lemma}

\begin{proof}\leavevmode
    \begin{enumerate}[\normalfont(i)]
    \item This follows immediately from Proposition~\ref{Pro. all successfully}. Indeed, if $w = (w_1,\dots,w_n) \in \Pf(n,m)$, then any permutation $w' = (w_{i_1}, \dots , w_{i_n})$ also lies in $\Pf(n,m)$, since both have the same nondecreasing rearrangement.
    
    \item Let $w \in [m]^n$, and define $p^{(w)} \in [m-1]^n$ by
    \[
    p^{(w)}_i =\begin{cases}
        1 & \text{if } w_i = 1, \\
        w_i - 1 & \text{if } w_i > 1.
    \end{cases}
    \]
    Since the first parking space is unavailable, all cars with preferences given by $w$ search for a space in $\{2,\dots,m\}$ in exactly the same way as cars with preferences $p^{(w)}$ search in $\{1,\dots,m-1\}$. Therefore, the parking process for $w$ with $l=1$ is equivalent to that for $p^{(w)}$ with $l=0$, in the sense that
    \[
    \out_{m,1}(w) = \out_{m-1,0}(p^{(w)}) + 1, \quad \text{where } \varnothing + 1 = \varnothing.
    \]
    Thus, all $n$ cars park successfully with preference word $w$ and $l=1$ if and only if all cars park successfully with preference word $p^{(w)}$ and $l=0$. The claim then follows from \ref{it:lemma. properties parking generalized I}.
    
    \item This follows immediately from \ref{it:lemma. properties parking generalized II}, since any prefix is itself a generalized parking function in $\Pf(j,m)$.
    
    \item If no such index exists, then no entry in the prefix equals $1$, and hence the parking process with $l=1$ coincides with that with $l=0$, contradicting the hypothesis. \qedhere
    \end{enumerate}
\end{proof}

\begin{proposition}\label{pro. y-parking sucesufull or last no}
    Let $w \in \Pf(n,m)$ be a generalized parking function. Consider the parking protocol with $m$ spaces and $l=1$ (i.e., the first space is unavailable). Then either all cars park successfully, or all cars except possibly the last one do.
\end{proposition}

\begin{proof}
    Assume that $w \in \Pf(n,m)$ does not yield a successful parking for all cars under the protocol with $l=1$. Let $j \in [n]$ be the index of the first car that fails to park.

    We first consider the case $w_j = 1$. Since cars only move to the right, if the $j$-th car fails to park with preference $1$, then all the remaining $m-1$ available spaces must already be occupied when it arrives; therefore since the previous $j-1$ cars have already parked, we must have $j\geq m$.
    Since $w \in \Pf(n,m)$ is only meaningful for $m \geq n$, and we obviously have that $j \leq n$, we conclude that it must be $j = n$. Thus, in this case, only the last car can fail to park.

    Now we consider $w_j > 1$. By Lemma~\ref{lemma. properties parking generalized}\ref{it:lemma. properties parking generalized IV}, we know that there exists $i < j$ such that $w_i = 1$. Suppose that $w_i = 1$ is the first occurrence of $1$ in $w$. Then consider the generalized parking function $w^{(1)} \in \Pf(n,m)$, which consists of swapping the entries $i$ and $j$, that is, $w^{(1)}_i = w_j$ and $w^{(1)}_j = w_i = 1$.
    
    We know by Lemma~\ref{lemma. properties parking generalized}\ref{it:lemma. properties parking generalized I} that $w^{(1)}$ is well-defined, and also, since the prefix $(w_1,\dots,w_j)$ means that the first car that fails to park is $w_j$, then by Lemma~\ref{lemma. properties parking generalized}\ref{it:lemma. properties parking generalized III}, the prefix 
    \[(w_1,\dots,w_{i-1},w_j,w_{i+1},\dots,w_{j-1},1),\] 
    there is also a car that fails to park. 

    Let $j_1$ be the index of the first car that fails to park on their desired location, we know that since $i$ was the first occurrence of $1$ in $w$, then for all $w^{(1)}_k$ with $k \leq i$ it holds that $w^{(1)}_k > 1$, so the parking protocol with $l=1$ is the same as with $l=0$, and since $w^{(1)}$ is a generalized parking function, all $i$ cars can park successfully, then $i < j_1 \leq j$.

    If $w^{(1)}_{j_1} = 1$, from the previous case we know that $j_1 = n$, since $j_1 \leq j$, then $j=n$, so the $n$-th car is the first to fail for $w$.

    Otherwise, apply the same reasoning again for $w^{(1)}_{j_1}$. At some time we will end with a chain $j_k \leq j_{k-1} \leq \cdots \leq j_{1} \leq j$. This forces $w^{(k)}_{j_k} = 1$, concluding that $j = n$.
\end{proof}

\begin{corollary}\label{col. all parking susefol or las not}
     Let $\yy=(y_1,\dots,y_k)\in \mathbb{Z}_{>0}^k$. Consider the modified parking protocol with $m = \sum_{i=1}^{n}y_i$ spaces and fix $l=1$. Then, under this protocol, either all cars park successfully, or all cars except the last one park successfully.
\end{corollary}
\begin{proof}
This follows immediately from the previous proposition and the inclusion
$\Pf_n(\yy) \subseteq \Pf(n,m)$ for $m=\sum_i y_i$.
\end{proof}

With this result from now on we can assume that for every $\yy$-parking function then the parking protocol with $l=1$, the only car that fails to park is the last one.

To prove Theorem \ref{The. magic for a y-parking functions}, as in the development of the block parking protocol, we will work with the magic polynomial
\[{f}_{\ss,l}(t,\yy) = n! \cdot \mathscr{M}({g}_{\ss,l}(t,\yy)) \text{ for all }\ss \in \mathcal{I}_n.\]

This formulation allows us to recover the role of the algebraic contributions introduced in Definition \ref{def. algebraic contributions}, and to relate them to concrete combinatorial objects.

\begin{guide*}{Completing Step (iii)}{}\label{guide2}
The proof of Theorem~\ref{The. magic for a y-parking functions}, and along with it also the completion of Step (iii), will be achieved as follows:

    \begin{enumerate}[Step (a)]
\item We refine Theorem \ref{The. magic for a y-parking functions} by introducing a finer statistic determined by a fixed subset of indices $L \subseteq [n]$. This refinement serves as an intermediate step to connect the problem with the algebraic contributions.

\item We define an algorithm that allows us to interpret the quantities $\alg_{l}^{(\alpha)}(\yy ; \pi,L)$ as the number of all possible outputs of this algorithm with inputs $\pi$ and $L$.

\item We show that this algorithm behaves well on the set of $\yy$-parking functions, allowing it to be used as a double counting tool.

\item Finally, by combining the tools developed in the previous steps, we conclude the equality obtained in Step (a), thereby completing the proof of the theorem.
\end{enumerate}
\end{guide*}

\subsection{Lucky statistic on a subset of lucky cars} \label{ssec: lucky subset lucky cars}
From now on, we fix $l \in \{0,1\}$ and let $\yy = (y_1,\dots,y_n)$ be a vector of positive integers, that is, $y_i \geq 1$ for each $i \in [n]$. Whenever required by the context, we set $m := \sum_{i=1}^{n} y_i$.

\begin{definition} Let $L \subseteq [n]$ be a subset of car indices. We define the following quantities:
\begin{itemize}
    \item The number of $\yy$-parking functions whose lucky set is contained in $L$:
    $$\lucky^{(\alpha)}_{l}(\yy; L) := \#\{ w \in \Pf_n(\yy) \mid \Lucky_{m,l}(w) \subseteq L \}.$$
    \item The number of $\yy$-parking functions whose lucky set is exactly $L$:
    $$\lucky^{(\beta)}_{l}(\yy; L) := \#\{ w \in \Pf_n(\yy) \mid \Lucky_{m,l}(w) = L \}.$$
\end{itemize}

\end{definition}

See Figure~\ref{Fig. example of parking 4 magic} and Table~\ref{tab: lucky parking alpha and beta} for an example for $n=2$.

\begin{table}[hbt]

$$
\begin{array}{crcrc} \hline 
    L & \operatorname{lucky}^{(\beta)}_{l}(\yy;L) & \mathbf{|_{y_1=2,y_2=1}} &\operatorname{lucky}^{(\alpha)}_{l}(\yy;L)  & \mathbf{|_{y_1=2,y_2=1}} \\ \hline
    \varnothing & 2 & 2& 2 & 2  \\
    \{1\} & 2y_1+y_2-2 & 3 & 2y_1+y_2 & 5\\
    \{2\} & y_1+y_2-2 & 1 & y_1+y_2 & 3\\
    \{1,2\} &y^2_1 + 2y_1 y_2 - 3y_1 - 2y_2 + 2 &2 &  y_1^2+2y_1y_2 &8\\ \hline
    \end{array}
$$
\caption{Number of $(y_1,y_2)$-parking functions with lucky set equal   to (contained in) $L$.}
\label{tab: lucky parking alpha and beta}
\end{table}

Recall that, given vectors of non-negative integers $\ss$ and $\yy$, the polynomial $f_{\ss,l}(t,\yy)$ is defined in Equation \eqref{eq. polynomial of algebraic lucky in}, and that the polynomials $g_{\Pf_n(\yy),l}(t)$ appear in Theorem \ref{The. magic for a y-parking functions}. Hence, we have the following observation.

\begin{remark}
    The polynomials $g_{\Pf_n(\yy),l}(t)$ can be expressed in terms of the polynomials $f_{\ss,l}(t,\yy)$ defined in Equation \eqref{eq. polynomial of algebraic lucky in}. More precisely, using Proposition \ref{prop. refine of f_s,y by pi and L}, we can write them in terms of algebraic contributions indexed by subsets $L \subseteq [n]$ and classical parking functions $\pi \in \Pf_n$:
    \begin{align}
        f_{\Pf_n(\yy),l}(t)
        &:= \sum_{\ss \in \mathcal{I}_{n}} f_{\ss,l}(t,\yy) \\
        &= \sum_{\ss \in \mathcal{I}_{n}} \sum_{L \subseteq [n]} \sum_{\pi \in \mulPerm{\ss}} \alg_{l}^{(\beta)}(\yy ; \pi, L)\, t^{|L|} \\
        &= \sum_{\pi \in \Pf_n} \sum_{L \subseteq [n]} \alg_{l}^{(\beta)}(\yy ; \pi, L)\, t^{|L|}.
    \end{align}
    Here, the regrouping of the sums over $\ss \in \mathcal{I}_n$ and $\pi \in \mulPerm{\ss}$ follows from the identity $\bigsqcup_{\ss}\mulPerm{\ss} = \Pf_n$.
\end{remark}

To prove Theorem \ref{The. magic for a y-parking functions}, it is therefore sufficient to establish the following proposition, which relates lucky statistics to algebraic contributions.

\begin{proposition}\label{Pro. refine igualdad dyck} 
Let $L \subseteq [n]$ be a subset of car indices. Then the following equality holds:
\begin{equation}
    \lucky^{(\beta)}_{l}(\yy; L) = \sum_{\pi \in \Pf_n}\alg_{l}^{(\beta)}(\yy ; \pi, L)
\end{equation}
\end{proposition}

By the inclusion--exclusion relation between the $(\alpha)$ and $(\beta)$ versions both for the lucky statistics and for the algebraic contributions Proposition \ref{Pro. refine igualdad dyck} is equivalent to its $(\alpha)$ formulation.

\begin{proposition}\label{Pro. refine igualdad dyck alpha} 
Let $L \subseteq [n]$ be a subset of car indices. Then the following equalities hold:
\begin{equation}\label{Eq. alpha 1}
    \lucky^{(\alpha)}_{l}(\yy; L) = \sum_{\pi \in \Pf_n}\alg_{l}^{(\alpha)}(\yy ; \pi, L)
\end{equation}
\end{proposition}

The advantage of working with the $(\alpha)$ formulation is that, by definition, the algebraic contributions $\alg_{l}^{(\alpha)}(\yy ; \pi,L)$ are always non-negative. This avoids cancellations in the summation and turns the problem into a double counting argument. 

\subsection{An algorithm preserving \texorpdfstring{$\yy$}{y}-parking functions} \label{ssec: algo preserving pf}

We recall the definition of the algebraic contributions $\alg_{l}^{(\alpha)}$ using the following notation.
Let $\pi \in \mathfrak{S}_{\ss}$ be a multiset permutation, let $L \subseteq [r]$ be a subset of car indices, and let $\yy \in \mathbb{Z}_{\geq 0}^{n}$ be a vector of non-negative integers. Remember the definition
\[o_{j,l}(\pi) := \mu(\pi)_j - 1 + \delta_{\pi_j,l},\]
where $\mu(\pi)_j$ denotes the partial multiplicity of the entry $\pi_j$ in $\pi$, defined for each $j \in [r]$ by
\[\mu(\pi)_j := \#\{ k \le j \mid \pi_k = \pi_j \}.\]
Note that, by definition, each quantity $o_{j,l}(\pi)$ is a non-negative integer.

Using the above notation, the algebraic contributions of type $(\alpha)$ can be written as
\[\alg_{l}^{(\alpha)}(\yy ; \pi,L)=\prod_{j \in L} y_{\pi_j}\;\cdot\;\prod_{j \in ([r]\setminus L)} o_{j,l}(\pi).\]

Moreover, this product admits a natural interpretation as a sequence of independent choices: for each $j \in [r]$, one chooses among $y_{\pi_j}$ possibilities if $j \in L$, and among $o_{j,l}(\pi)$ possibilities if $j \notin L$. This viewpoint motivates the construction of the following algorithm.

\begin{algorithm}\label{Alg. count alg by lucky set}
    Let $\ss,\yy \in \mathbb{Z}_{\geq 0}^{n}$ be vectors of non-negative integers, and fix $l \in \{0,1\}$.  
    Given a multiset permutation $\pi \in \mathfrak{S}_{\ss}$ and a subset $L \subseteq [r]$, where $r = \sum_{i=1}^{n} s_i$, the algorithm produces as output a word $\sigma \in [m]^r$ such that $\Lucky_{m,l}(\sigma) \subseteq L$, where $m = \sum_{i=1}^{n} y_i$.
    
    The procedure is as follows:
    \begin{enumerate}[\normalfont (1)]
        \item \textbf{Initialization.}  
        For each $i \in [n]$, define the occupied sets $O_i^{(0)} = \emptyset$.  
        If $l=1$, initialize $O_1^{(0)} = \{1\}$.  
        And the blocks associated to $\yy$ by
        $$Y_i = \{ y_1 + \cdots + y_{i-1} + 1, \dots, y_1 + \cdots + y_i \}.$$
        
        \item \textbf{Preference assignment.}  
        For each step $j \in [r]$, determine $\sigma_j$ as follows:
        \begin{itemize}
            \item If $j \in L$, choose $\sigma_j \in Y_{\pi_j}$.
            \item If $j \notin L$, choose $\sigma_j \in O^{(j-1)}_{\pi_j}$.
        \end{itemize}
    
        \item \textbf{Parking protocol simulation.}  
        Apply the parking protocol to the prefix $\sigma^{(j)} := (\sigma_1,\dots,\sigma_j)$, obtaining the output $\out_{m,l}(\sigma^{(j)})$.
    
        \item \textbf{Update of occupied sets.}  
        For each $i \in [n]$, update
        $$ O^{(j)}_{i} :=\{ \out_{m,l}(\sigma^{(j)})_k \mid k \le j,\ \pi_k = i,\ \out_{m,l}(\sigma^{(j)})_k \neq \varnothing \}.$$
    
        \item \textbf{Iteration.}  
        Repeat Steps (2)–(4) until all indices $j = 1,\dots,r$ have been processed.
    \end{enumerate}
\end{algorithm}

\begin{remark}\label{Rmk. 1 remark algortim}
    At each step $j$, the choice between $Y_{\pi_j}$ and $O^{(j-1)}_{\pi_j}$ depends on whether $j \in L$ or not.  
    Since the sets $O^{(j-1)}_i$ record the positions already taken by cars $k < j$ and such that $\pi_k = i$, then any choice from $O^{(j-1)}_{\pi_j}$ guarantees that the $j$-th car is not lucky.  
    Therefore, the algorithm always produces words $\sigma$ satisfying $\Lucky_{m,l}(\sigma) \subseteq L$.
\end{remark}

\begin{remark}\label{Rmk. key invariant alg}
Assume that for all steps $j' < j$, every choice made by the algorithm produces a prefix
$\sigma^{(j')}$ in which all cars successfully park.
Then, for each $j' < j$, the set $O^{(j'-1)}_{\pi_{j'}}$ has cardinality
\[\# O^{(j'-1)}_{\pi_{j'}}  \;=\; o_{j',l}(\pi).\]
To see this, we proceed as follows. Since no car fails before step $j$, each car $k < j$ with $\pi_k = i$ occupies a distinct parking space.
Thus, the set $O^{(j-1)}_i$ contains exactly one element for each previous occurrence of $i$ in $\pi$.
If $i \neq l$, this yields 
$$\# O^{(j-1)}_i = \mu(\pi)_j - 1.$$
If $i = l$ and $l=1$, the initially blocked space contributes one additional element, giving 
$$\# O^{(j-1)}_1 = \mu(\pi)_j - 1 + 1.$$
In both cases, we obtain $\# O^{(j-1)}_{\pi_j} = o_{j,l}(\pi)$.
\end{remark}

\begin{remark}\label{Rmk. 2 remark algortim}
    Assume that, for all steps $j < r$, every possible choice produces a prefix $\sigma^{(j)}$ in which all cars successfully park.
    By Remark \ref{Rmk. key invariant alg}, at each step $j$ the number of available choices for $\sigma_j$
    is exactly $y_{\pi_j}$ if $j \in L$, and exactly $o_{j,l}(\pi)$ if $j \notin L$.
    Therefore, the total number of distinct outputs of the algorithm for a fixed input
    $\pi \in \mulPerm{\ss}$ is precisely $\alg_{l}^{(\alpha)}(\yy ; \pi,L)$.
\end{remark}

\begin{remark}\label{Rmk. role of last car}
A crucial feature of Algorithm \ref{Alg. count alg by lucky set} is that the counting of outputs only depends on the behavior of the parking protocol up to step $r-1$.
Once all cars indexed by $j < r$ have successfully parked, the number of available choices at each step is already fixed and realizes the algebraic contribution $\alg_{l}^{(\alpha)}(\yy ; \pi,L)$.

Consequently, whether the last car parks successfully or not does not affect the total count of outputs of the algorithm, since no further updates of the occupied sets are required.
\end{remark}

Algorithm \ref{Alg. count alg by lucky set} induces a partition of the set of generalized parking functions $\Pf(r,m)$ whose lucky set is contained in $L$. This partition depends on the vector $\yy$.

Given a generalized parking function $\sigma \in \Pf(r,m)$ such that $\lucky_{m,l}(\sigma )\subseteq L$, it is possible to reconstruct deterministically the word $\pi \in [n]^r$ that would have produced $\sigma$ via Algorithm~\ref{Alg. count alg by lucky set}. This \textit{reconstruction procedure} is described as follows:

\begin{itemize}
    \item Set $\pi _1=i$ such that $\sigma _1\in Y_i$.
    \item For each step $j=2,\dots ,r$, determine $\pi _j$ according to the following rule:
    \begin{itemize}
        \item If $j\in L$, then set $\pi _j=i$ such that $\sigma _j\in Y_i$.
        \item If $j\notin L$, then set $\pi _j=i$ such that $\sigma _j\in O_i^{(j-1)}$, where the sets $O_i^{(j-1)}$ are updated by simulating the parking protocol on the prefix $\sigma^{(j)}=(\sigma _1,\dots ,\sigma _j)$, the same way as the Algorithm \ref{Alg. count alg by lucky set}:
        $$O^{(j)}_{i} := \{\out_{m,l}(\sigma^{(j)})_k \mid k \leq j, \ \pi_k = i\}.$$
    \end{itemize}
\end{itemize}

At the end of this process, one obtains a word $\pi \in [n]^r$ which, when used as input to Algorithm~\ref{Alg. count alg by lucky set}, produces $\sigma$, with the correct choices produces exactly the generalized parking function $\sigma$. This procedure is deterministic. Therefore we have that the following property:

\begin{lemma}\label{Lemma. partition parking by algorim}
Let $\sigma \in \Pf(r,m)$ be a generalized parking function such that $\Lucky_{m,l}(\sigma )\subseteq L$, and let $\yy \in \Z_{\geq 0}^{n}$. Then there exists a unique word $\pi \in [n]^r$ such that the Algorithm \ref{Alg. count alg by lucky set}, when applied to $\pi$ , produces $\sigma$ . Moreover, if two such words are distinct, then the sets of generalized parking functions they produce via the algorithm are disjoint.
\end{lemma}

The input to the Algorithm \ref{Alg. count alg by lucky set} is a multiset permutation $\pi \in \mathfrak{S}_{\ss}$, where $\ss$ is a weak composition of $r$, encoding the multiplicities of the symbols in $\pi$ (each $s_i$ counts the number of occurrences of $i$ in $\pi$).

We now introduce an auxiliary tracking mechanism during the execution of the algorithm. 

Let $\ss^{(0)}:=\ss$ be the initial composition. At each step $j\in [r]$, we identify the interval $Y_i$ such that $\sigma _j\in Y_i$.
\begin{itemize}
    \item If $i=\pi _j$, then no update is performed, and we set $\ss^{(j)}:=\ss^{(j-1)}$.
    \item If $i\neq \pi_j$, then we update the composition vector as follows:
    $$\ss^{(j)} = \ss^{(j-1)} + e_i - e_{\pi_j}$$
\end{itemize}
Here $e_i$ is the standard unit vector where $(e_i)_j = \delta_{i,j}$. 

This update reflects the fact that the symbol $\sigma_j$ belongs to the interval $Y_i \not=Y_{\pi_j}$. In such cases, the number of occurrences of the symbols $\pi_j$ and $i$ is reassigned in the weak composition $\ss^{(j)}$.

As a result, at each step $j$, the prefix $\sigma^{(j)} := (\sigma_1,\dots,\sigma_j)$ is the prefix of some $\yy$-extended permutation in $\ExtPerm{\ss^{(j)}}{\yy}$. In particular, at the end of the algorithm, the final word $\sigma$ is in $\ExtPerm{\ss^{(r)}}{\yy}$. \\

Conversely, given an output word $\sigma$, we can reconstruct the multiset permutation $\pi$ via the same procedure. We initialize a composition vector $\ss'^{(0)} = (\ss'^{(0)}_1,\dots,\ss'^{(0)}_n)$ as follows:
$$\ss'^{(0)}_i = \#\{j \in [r] \mid \sigma_j \in Y_i\}$$

Therefore $\sigma \in \ExtPerm{\ss'^{(0)}}{\yy}$. As we reconstruct the multiset permutation $\pi$  step by step, we update the composition vector $\ss'^{(j)}$ according to the following rule:
\begin{itemize}
    \item If $\pi _j=i$ such that $\sigma _j\in Y_i$, then no update is performed, and $\ss'^{(j)}:=\ss'^{(j-1)}$.
    \item If $\pi _j\not= i$ such that $\sigma _j\in Y_i$, then we update:
    $$\ss'^{(j)} = \ss'^{(j-1)} - e_i + e_{\pi_j}$$
\end{itemize}
This update reflects the reassignment of occurrences from the symbol $i$ to the symbol $\pi_j$, reversing the transformation performed during the forward execution of Algorithm \ref{Alg. count alg by lucky set}.

At the end of the reconstruction procedure, the resulting word $\pi$  satisfies $\pi \in \mulPerm{\ss'^{(r)}}$, where $\ss'^{(r)}$ is the final composition obtained after all updates.

\begin{remark}\label{Remark naturlay of ss}
    During the forward execution of Algorithm \ref{Alg. count alg by lucky set}, suppose that at some step $j$ the symbol $\sigma_j \in Y_i$ with $i \neq \pi_j$. Then necessarily $j \notin L$, since the algorithm prescribes that if $j \in L$, the choice $\sigma_j$ must belong to $Y_{\pi_j}$.
    
    Thus, $\sigma_j$ must have been chosen from the set $O^{(j-1)}_{\pi_j}$, which consists of parking outcomes of previous cars $k<j$ such that $\pi_k=\pi_j$.
    By the parking protocol, any such outcome satisfies
    $$\out_{m,l}(\sigma^{(j-1)})_k > \sigma_k,$$
    since cars only move to the right when their preferred space is occupied.
    
    Consequently, if $\sigma_j \in Y_i$ with $i \neq \pi_j$, then necessarily $\pi_j < i$.
    This implies that the update
    $$\ss_{\pi_j}^{(j)} = \ss_{\pi_j}^{(j-1)} - 1, \qquad  \ss_i^{(j)} = \ss_i^{(j-1)} + 1$$
    transfers a unit from a smaller index $\pi_j$ to a larger index $i$.

    In contrast, during the reconstruction procedure, the update
    $$\ss'^{(j)}_{\pi_j} = \ss'^{(j-1)}_{\pi_j} + 1, \qquad \ss'^{(j)}_i = \ss'^{(j-1)}_i - 1$$
    reverses this process, transferring a unit from a larger index $i$ to a smaller index $\pi_j$.
    
    This monotonicity of the updates will be essential in establishing the dominance condition defining the $\mathcal{I}_n$.
\end{remark}

\begin{lemma}\label{Lemma. preserve dominance}
    Let $\yy=(y_1,\dots,y_n)$ be a vector of positive integers. 
    Suppose that at step $j$ of Algorithm~\ref{Alg. count alg by lucky set} the composition vector $\ss^{(j)}$ is updated. 
    
    If the previous composition $\ss^{(j-1)}$ satisfies
    $$\sum_{l=1}^i s_l^{(j-1)} \ge i \quad \text{for all } i=1,\dots,n,$$
    then the updated composition $\ss^{(j)}$ also satisfies
    $$\sum_{l=1}^i s_l^{(j)} \ge i \quad \text{for all } i=1,\dots,n.$$
\end{lemma}

\begin{proof}
    An update occurs only when $\sigma_j \in Y_i$ with $i \neq \pi_j$. In this case,
    $$\ss_{\pi_j}^{(j)} = \ss_{\pi_j}^{(j-1)} - 1, \qquad \ss_i^{(j)} = \ss_i^{(j-1)} + 1,$$
    with $\pi_j < i$, and all other entries unchanged.
    
    For $k < \pi_j$, the partial sums remain unchanged. For $k \ge i$, we also have
    $$\sum_{l=1}^k \ss_l^{(j)} = \sum_{l=1}^k \ss_l^{(j-1)},$$
    so the inequality is preserved. Thus it suffices to consider $\pi_j \le k < i$. For such $k$,
    $$\sum_{l=1}^k \ss_l^{(j)} = \sum_{l=1}^k \ss_l^{(j-1)} - 1.$$
    Hence, a violation could only occur if
    $$\sum_{l=1}^k \ss_l^{(j-1)} = k.$$
    
    We now rule out this possibility. Since $\sigma_j \in Y_i$ with $i>\pi_j$, the $j$-th car must have passed through all blocks $Y_{\pi_j+1},\dots,Y_{i-1}$ without parking. As each $y_k \ge 1$, every block $Y_k$ contains at least one position, which must therefore already be occupied when car $j$ arrives.
    
    Thus, for each $k \in [\pi_j,i-1]$, there exists a car that parked in $Y_k$ before step $j$. Any such car must correspond to an index $\le k$, and therefore contributes to the sum $\sum_{l=1}^k \ss_l^{(j-1)}$. Hence
    $$\sum_{l=1}^k \ss_l^{(j-1)} \ge k+1.$$
    After subtracting one unit, we obtain
    $$\sum_{l=1}^k \ss_l^{(j)} \ge k,$$
    as required.
\end{proof}

\begin{proof}[Proof of Proposition~\ref{Pro. refine igualdad dyck alpha} \eqref{Eq. alpha 1}]
    Let $\sigma$ be a $\yy$-parking function such that
    $\Lucky_{m,l}(\sigma)\subseteq L$. Then there exists a weak composition $\ss'\in \mathcal{I}_n$
    such that $\sigma\in \ExtPerm{\ss'}{\yy}$.
    
    Apply the reconstruction procedure to $\sigma$, and track the evolution of the
    composition vector $\ss'^{(j)}$. By Remark~\ref{Remark naturlay of ss}, each update transfers a unit from a
    larger index to a smaller one. In particular, for every $i$,
    $$\sum_{l=1}^i \ss'^{(j)}_l \ge \sum_{l=1}^i \ss'^{(j-1)}_l.$$
    Since the initial composition $\ss'\in\mathcal{I}_n$ satisfies the dominance condition, it follows that all intermediate compositions $\ss'^{(j)}$ also lie in $\mathcal{I}_n$.
    In particular, the final composition $\ss_0 := \ss'^{(r)}$ belongs to $\mathcal{I}_n$.
    
    Consequently, for every $\sigma$ with $\Lucky_{m,l}(\sigma)\subseteq L$, there exists a unique $\ss_0\in\mathcal{I}_n$ and a unique $\pi\in\mulPerm{\ss_0}$ such that Algorithm~\ref{Alg. count alg by lucky set}
    produces $\sigma$ from $\pi$.
    
    By Lemma~\ref{Lemma. partition parking by algorim}, the algorithm is injective, hence
    $$\lucky^{(\alpha)}_{l}(\yy; L) \le \sum_{\ss\in\mathcal{I}_n}\sum_{\pi\in\mulPerm{\ss}} \alg_{l}^{(\alpha)}(\yy ; \pi,L).$$
    
    Conversely, let $\sigma$ be an output of Algorithm~\ref{Alg. count alg by lucky set} applied to some $\pi\in\mathfrak{S}_{\ss}$ with $\ss\in\mathcal{I}_n$.
    By Lemma~\ref{Lemma. preserve dominance}, the final composition $\ss_1 := \ss^{(r)}$ also lies in $\mathcal{I}_n$,
    and $\sigma\in\ExtPerm{\ss_1}{\yy}$. Hence $\sigma$ is a $\yy$-parking function.
    
    Moreover, for each $j$, the prefix $\sigma^{(j)}$ is a prefix of a $\yy$-parking function. By Corollary~\ref{col. all parking susefol or las not}, all cars park successfully except possibly the last one. Therefore, the hypothesis of Remark~\ref{Rmk. 2 remark algortim} holds, and $\alg_{l}^{(\alpha)}(\yy ; \pi,L)$ counts exactly all outputs corresponding to $\pi$.
    
    This shows that
    $$\lucky^{(\alpha)}_{l}(\yy; L) = \sum_{\ss\in\mathcal{I}_n}\sum_{\pi\in\mulPerm{\ss}} \alg_{l}^{(\alpha)}(\yy ; \pi,L),$$
    as claimed.
\end{proof}

The hypothesis that $y_i \geq 1$ is necessary to guarantee the dominance condition in Lemma~\ref{Lemma. preserve dominance} is preserved. The following example shows that when some $y_i=0$, the dominance condition may fail during the execution of the algorithm. In particular, Lemma~\ref{Lemma. preserve dominance} no longer holds, and the algorithm may produce words that are not $\yy$-parking functions.

\begin{example}\label{ex:failure algorithm zero yi}
We revisit Example~\ref{ex:PS 0 ys that is not magic positive} from the perspective of Algorithm~\ref{Alg. count alg by lucky set}. Fix $l = 1$ and recall that for $n=3$ and $\yy = (1,0,2)$, the magic positivity fails. We now show that this failure is reflected in the application of the algorithm.

The blocks associated to $\yy$ are: 
$$Y_1 = \{1\}, \quad Y_2 = \emptyset, \quad Y_3 = \{2,3\}.$$
Let $\pi = (1,1,1) \in \Pf_3$ and $L = \emptyset$. We run the algorithm always choosing the largest available value.
Initialize $O_1^{(0)} = \{1\}$ (given that $l = 1$) and $O_2^{(0)} = O_3^{(0)} = \emptyset$.

\begin{itemize}
    \item Step $j=1$. The only choice is $\sigma_1 = 1$. Then $\out_{3,1}(\sigma^{(1)}) = (2)$, so $O_1^{(1)} = \{1,2\}$.
    \item Step $j=2$. Choose $\sigma_2 = 2$, hence $\sigma^{(2)} = (1,2)$ and $\out_{3,1}(\sigma^{(2)}) = (2,3)$. Thus $O_1^{(2)} = \{1,2,3\}$.
    \item Step $j=3$. Choose $\sigma_3 = 3$, obtaining $\sigma = (1,2,3)$.
\end{itemize}

We now show that $\sigma$ is not a $\yy$-parking function. Tracking the composition vector, we start from $\ss^{(0)} = (3,0,0)$ and obtain $\ss^{(3)} = (1,0,2)$.

This composition violates the dominance condition since $1 + 0 < 2$, so $\ss^{(3)} \notin \mathcal{I}_3$. Therefore, $\sigma \notin \Pf_3(\yy)$.

\end{example}

This example illustrates that the assumption $y_i \ge 1$ is used to guarantee that every block $Y_i$ contains at least one available position, which is crucial in the proof of Lemma~\ref{Lemma. preserve dominance}.

\section{Final Remarks} \label{sec: final remarks}

\subsection{Another parking protocol for \texorpdfstring{$\yy$}{y}-parking functions} \label{sec: other parking protocols y-parking}

After we finished the first version of this manuscript \cite{avila-ferroni-morales-abstract}, we discovered that in \cite{ferreri2025enumeratingvectorparkingfunctions}, Ferreri, Harris, Martinez, and Swartz independently studied another parking protocol for ${\bf y}$-parking functions that exhibits some similarities to the one treated in the present paper. We know that ${\bf y}$-parking functions are equivalent to ${\bf u}$-parking functions under the identification $u_i = y_1 + \cdots + y_i$. 

Given $\yy=(y_1,\ldots,y_n) \in \mathbb{Z}^n_{\geq 0}$ with $m=y_1+\cdots+y_n$, and the word $w \in [m]^n$, they interpret $w = (w_1, \dots, w_n)$ as the parking preferences of $n$ cars that try to park in $m$ available spaces, each space $j \in [m]$ with capacity $m_{j}({\bf u})$ the multiplicity of $j$ in ${\bf u}$. A car can park at a space if there is still capacity. 

Each car $i \in [n]$ tries to park in its preferred space $w_i$.  If that space $w_i$ is already occupied by $m_{w_i}({\bf u})$ cars, the car moves to the next available space $w_i+1$. If no available space is found after reaching the end, the car is considered to have failed to park. In this setting, the words $w$ that lead to all cars parking successfully are exactly the ${\yy}$-parking functions in $\Pf_n(\yy)$.  

In contrast, our parking protocol for ${\bf y}$-parking functions considers $m = \sum_{i=1}^n y_i$ linearly ordered spaces with unit capacity. While every ${\bf y}$-parking function parks successfully under our protocol (see Proposition~\ref{Pro words on Dyck path successfully park}), the converse does not hold in general: there exist preference sequences that park successfully but are not ${\bf y}$-parking functions. For example, if ${\bf y} = (1,2)$, the word $w=(2,3)$ parks successfully in $m=3$ spaces, but it is not a ${\bf y}$-parking function.  Note that both protocols agree in the classical case when $\yy=(1,\ldots,1)$.

As a consequence, the notion of ``lucky cars'' depends on the chosen protocol and is not preserved under this correspondence. In particular, the associated statistics are not equidistributed 
However, when $\ell=1$ in our protocol, calculations suggest that both ours and their lucky statistic are co-equidistributed (see Example~\ref{ex:stat codistributed}). For  $\yy \in \mathbb{Z}^n_{> 0}$, let $L \subseteq [n]$, denote as $\operatorname{lucky}_{\text{FHMS}}(\yy;L)$ the number of $\yy$-parking functions with lucky cars $L$ under the FHMS protocol and for a $\yy$-parking function $\sigma$, let $\lucky_{\text{FHMS}}(\yy;\sigma)$ be the number of lucky cars under the FHMS protocol. 

\begin{conjecture} \label{conj: equidistribution both protocols}
For $\yy \in \mathbb{Z}^n_{> 0}$, we have that 
    \begin{equation} \label{eq: lucky set equiv}
    \operatorname{lucky}_{\text{FHMS}}(\yy; L) = \lucky^{(\beta)}_{1}(\yy;[n]\setminus \overline{L}),
    \end{equation}
    where $\overline{L} := \{n +1- i  \ | \ i \in L\}$.
In particular, we have that  
\begin{equation} \label{eq: lucky number equiv}
\sum_{\sigma \in \Pf_n(\yy)}t^{\lucky_{\text{FHMS}}(\yy; \sigma)} \,=\, \sum_{\sigma \in \Pf_n(\yy)}t^{n-\lucky_{m,1}(\sigma)}.
\end{equation}
\end{conjecture}

\begin{example} \label{ex:stat codistributed}
For ${\bf y}=(1,2)$ (so ${\bf u}=(1,3)$), the ${\bf y}$-parking functions and the corresponding set and number of lucky cars in our setting for $\ell=0,1$ and in the setting of \cite{ferreri2025enumeratingvectorparkingfunctions} are given in Table~\ref{table:lucky stats}. The generating polynomials for the number of lucky cars are $4t^2 + t$, $3t+2$, and $2t^2 + 3t$, respectively. Note that the first and last notions are not equidistributed, but the second and third are (up to taking the complement), illustrating \eqref{eq: lucky number equiv} in Conjecture~\ref{conj: equidistribution both protocols}. Note also how the data in the third and fifth columns illustrate \eqref{eq: lucky set equiv} in Conjecture~\ref{conj: equidistribution both protocols}. 

\begin{table}
$$\begin{array}{c|cc|cc|cc}\hline
\Pf_n(\yy) & \text{Lucky}_{m,0} & \text{lucky}_{m,0} & \text{Lucky}_{m,1} & \text{lucky}_{m,1} &  \text{Lucky}_{\text{FHMS}} &  \text{lucky}_{\text{FHMS}} \\ \hline
(1, 2) & \{1,2\} & 2 & \emptyset & 0 & \{1\}   & 1\\ 
(1, 3) & \{1,2\} & 2 & \{2\}     & 1 & \{1,2\} & 2\\ 
(2, 1) & \{1,2\} & 2 & \{1\}     & 1 & \{2\}   & 1\\ 
(3, 1) & \{1,2\} & 2 & \{1\}     & 1 & \{1,2\} & 2\\ 
(1, 1) & \{1\}   & 1 & \emptyset & 0 & \{1\}   & 1\\ 
\hline\end{array}$$
\caption{Comparison of the lucky statistic and that in \cite{ferreri2025enumeratingvectorparkingfunctions} on the ${\bf y}$-parking functions for ${\bf y}=(1,2)$. In this case, where ${\bf u}=(1,3)$ and thus $m_1({\bf u})=m_3({\bf u})=1$ and $m_2({\bf u})=0$.} \label{table:lucky stats}
\end{table}

\end{example}

\begin{example}
For ${\bf y}=(4,3,5)$ (so ${\bf u}=(4,7,12)$), the number of ${\bf y}$-parking functions with lucky cars  given by $L\subset [3]$ under the FHMS protocol and our protocol for $\ell=1$ are given in Table~\ref{table:larger example conjecture}, illustrating \eqref{eq: lucky set equiv} in Conjecture~\ref{conj: equidistribution both protocols}. 

\begin{table}
\[
\begin{array}{ccccc}\hline
L & \overline{L} & [3]\setminus \overline{L}&  \text{Lucky}_{\text{FHMS}}(\yy; L) & \text{Lucky}^{(\beta)}_{1}(\yy; L) \\ \hline
\varnothing &\varnothing&\left\{1,2,3\right\} &474 & 6 \\
\left\{1\right\} & \left\{3\right\}&\left\{1,2\right\}&168 & 40 \\
\left\{2\right\} & \left\{2\right\}&\left\{1,3\right\}&115 & 25 \\
\left\{1, 2\right\} & \left\{2,3\right\}&\left\{1\right\}&40 & 168 \\
\left\{3\right\} & \left\{1\right\}&\left\{2,3\right\}&70 & 18 \\
\left\{1, 3\right\} &\left\{1,3\right\}&\left\{2\right\}& 25 & 115 \\
\left\{2, 3\right\} &\left\{1,2\right\} &\left\{3\right\}&18 & 70 \\
\left\{1, 2, 3\right\} &\left\{1,2,3\right\}&\varnothing& 6 & 474 \\ 
\hline\end{array}
\]
\caption{For $\yy=(4,3,5)$, the number of $\yy$-parking functions with lucky cars $L\subset [3]$ under the FHMS protocol and under our protocol for 
$\ell=1$. The sets $\overline{L}$ and $[3]\setminus \overline{L}$ are also included.}
\label{table:larger example conjecture}
\end{table}
\end{example}

\subsection{A curious identity for the case of generalized parking functions \texorpdfstring{$\Pf(n,m)$}{Pf(n,m)}}

A curious consequence outside the polytope context appears in the case $l=0$. In this situation, the parking protocol coincides with that studied by Kenyon and Yin \cite{kenyonparkingfunctions}. They show that for generalized parking functions $\Pf(n,m)$, the generating function of the statistic $\lucky_{m,0}$ is given by
\cite[Eq.~(2.20)]{kenyonparkingfunctions}
\begin{equation}\label{eq. lucky polynomial}
    \sum_{w \in \Pf(n,m)}t^{\lucky_{m,0}(w)} = (m - n + 1)t\prod_{i=1}^{n-1}(i+(m-i+1)t).
\end{equation}

Recall that when $\yy = (m-n+1,1,\dots,1)$, one has $\Pf_n(\yy)=\Pf(n,m)$. Therefore, Theorem~\ref{The. magic for a y-parking functions} implies that the left-hand side of \eqref{eq. lucky polynomial} can be expressed explicitly as a sum indexed by Dyck paths $\mathcal{I}_n$. Comparing both expressions yields the following identity.

\begin{proposition}
    For $m \geq n$, the following identity holds:
    $$\sum_{s \in \mathcal{I}_n}\bbinom{(m-n + 1)t}{s_1}\prod_{i=2}^{n}\bbinom{t}{s_i} = \frac{m - n + 1}{m+1}\bbinom{(m+1)t}{n}$$
\end{proposition}
\begin{proof}
We can replace the left-hand side of \eqref{eq. lucky polynomial} with the sum $\sum_{\ss \in \mathcal{I}_{n}} f_{\ss,0}(t,\yy)$ by Theorem \ref{The. magic for a y-parking functions}, therefore we have that:
\begin{align*}
    \sum_{\ss \in \mathcal{I}_{n}} f_{\ss,0}(t,\yy) &= (m - n + 1)t\prod_{i=1}^{n-1}(i+(m-i+1)t) \\
    &= \frac{(m - n + 1)}{m+1}\prod_{i=0}^{n-1}(i+(m+1-i)t)\\
    &= \frac{(m - n + 1)}{m+1}\prod_{i=0}^{n-1}(n-1 - i+(m + 1 +i - (n-1))t).
\end{align*}
Note that the right-hand side is a polynomial of the form $n!\cdot \mathscr{M}(p_{m+1,n-1,n}(t))$ as in Corollary~\ref{coro:magic-binomials}, and the left-hand side is the magic polynomial of $g_{\Pf_n(\yy),0}(t)$, We conclude the identity by applying the transformation $\mathscr{M}^{-1} / n!$ to both sides of the equality.
\end{proof}

\subsection{Failure of magic positivity for \texorpdfstring{$\yy$}{y}-generalized permutohedra}

As we mentioned in the introduction, there are counterexamples to the magic positivity of Ehrhart polynomials of $\yy$-generalized permutohedra. The following example provides one such instance of this failure of magic positivity. 

\begin{example} \label{ex:y perm not magic positive}
For $m=3$, $n=7$, and $H\subset K_{3,5}$ with edges $$(1,1),(1,2),(2,2),(2,3),(2,4),(3,4),(3,5),(3,7),
$$
the generalized permutohedron $\mathcal{P}_H({\bf 1}) = \Delta_{\{1,2\}}+\Delta_{\{2,3,4\}}+\Delta_{\{4,5\}}$ has Ehrhart polynomial $\ehr_{\mathcal{P}_H}(t) = \frac{1}{2} t^{4} + \frac{5}{2} t^{3} + \frac{9}{2} t^{2} + \frac{7}{2} t + 1$ which is not magic positive, since $\mathscr{M}(\ehr_{\mathcal{P}_H}(t)) = -\frac{1}{2} t + 1$.   
\end{example}

\subsection{Magic positivity and determinants} \label{sec:magic kreweras determinants}

Since the Ehrhart polynomial of the Pitman--Stanley polytope admits a determinantal formula and since transformation $\mathscr{M}_d$ interacts well with products, our main result can be recast so as to assert the following curious positivity property of a determinant.

\begin{corollary}
    For $\mathbf{y} \in \mathbb{Z}^n_{>0}$,  the following polynomial has non-negative coefficients:
    \[ p_{\mathbf{y}}(t) = \det \left[ \frac{1}{(j-i+1)!} \prod_{\ell=0}^{i-j} (1-n+t(y_1+\cdots+y_{i} - 1 + n))\right]_{i,j=1}^n\]
\end{corollary}

\begin{proof}
  We apply the linear map $\mathscr{M}$ on the determinant formula \eqref{eq:detEhrhartPS} for $\ehr_{\Pi_n(\yy)}(t)$ and apply Lemma~\ref{lemma:magic-multiplicative} repeatedly and obtain
  \[
  \mathscr{M}(\ehr_{\Pi_n(\yy)}) \,=\, 
  \det\left[\mathscr{M}\binom{t\cdot(y_1+\cdots +y_{i})+1}{j-i+1}\right]_{i,j=1}^n. \]
  Then, using Corollary~\ref{coro:magic-binomials} yields the result directly.
\end{proof}

As explained in \cite{Pitman_Stanley_1999}, the determinant formula for the Ehrhart polynomial of $\Pi_n(\yy)$ is a consequence of Kreweras' formula for the number of {\em plane partitions} of straight shape with entries $1,2$ where the parts of the partition are scaled. Recall that a \emph{plane partition} of a skew shape is the number of fillings of the cells of the Young diagram of the shape such that the rows and columns are weakly decreasing. A similar but unrelated positivity phenomenon of the Kreweras formula was recently found in \cite{ferroni-morales-panova}, namely the number of plane partitions of a fixed skew shape with entries $1,\ldots,t$ is a polynomial with nonnegative coefficients in $t$. Precisely, for every skew shape $\lambda/\mu$, the polynomial
    \[ \Omega_{\lambda/\mu}(t) :=  \det\left[ \binom{t-1 + \lambda_i-\mu_j }{\lambda_i-\mu_j-i+j}
\right]_{i,j=1}^{\ell}\]
has nonnegative coefficients.

We are thus motivated to inquire if there exists a larger, more general, framework which can be used to tackle simultaneously the positivity of the coefficients of the polynomials $p_{\yy}(t)$ and $\Omega_{\lambda/\mu}(t)$. In other words, we ask for a class of univariate polynomials that arise from determinants, have nonnegative coefficients, and which can be specialized to the polynomials $p_{\yy}(t)$ and $\Omega_{\lambda/\mu}(t)$.

\subsection{Skew and Generalized Pitman--Stanley polytopes}

In their original paper \cite{Pitman_Stanley_1999}, Pitman and Stanley introduced two generalizations of the polytope $\Pi_n(\yy)$. The first one is the {\em skew Pitman--Stanley polytope} defined as follows, for $\zz,\yy \in \mathbb{Z}^n_{\geq 0}$, let 
\[
\Pi_n(\yy;\zz) = \left\{ \mathbf{x} \in \mathbb{R}_{\geq 0}^n : \enspace\sum_{i=1}^k z_i \leq \sum_{i=1}^k x_i \leq \sum_{i=1}^k y_i \text{ for each  $k=1,\ldots,n$}\right\}.
\]

The original $\Pi_n(\yy)$ corresponds to the case $\zz = (0,\ldots,0)$. The name comes from the fact that lattice points of this polytope correspond to plane partitions of the skew shape $(x_1+\cdots+x_n,\ldots,x_1)/(z_1+\cdots+z_n,\ldots,z_1)$  with entries $1,2$. The normalized volume of $\Pi_n(\yy;\zz)$ is the number of certain restricted $\yy$-parking functions \cite{gen_PS2}. Nevertheless, as the following example shows, not all skew Pitman--Stanley polytopes have magic positive Ehrhart polynomials.

\begin{example}
For $n=6$, $\yy=(1,1,1,1,1,1)$ and $\zz = (0,0,0,1,1,1)$, the Ehrhart polynomial is $\ehr_{\Pi_6(\yy,\zz)}(t) = \frac{769}{36} t^{6}+\frac{238}{3} t^{5}+\frac{2141}{18} t^{4}+\frac{280}{3} t^{3}+\frac{1483}{36} t^{2}+\frac{59}{6} t +1$ which is not magic positive, since $\mathscr{M}(\ehr_{\Pi_6(\yy,\zz)}(t)) = -\frac{1}{6} t^{5} + \frac{25}{9} t^{4} + \frac{62}{9} t^{3} + \frac{253}{36} t^{2} + \frac{23}{6} t + 1$.

\end{example}

We note that skew Pitman--Stanley polytopes are special cases of \emph{polypositroids} (as defined in the work of Lam and Postnikov \cite{lam-postnikov}). In work of Ferroni, Jochemko, and Schr\"oter \cite[Conjecture~6.1]{ferroni-jochemko-schroter} it is conjectured that $0/1$-polypositroids (i.e., classical positroids) are Ehrhart positive. If this Ehrhart positivity conjecture is true, it would be reasonable to expect it to also hold for polypositroids, and thus in particular for all $\Pi_n(\yy,\zz)$.

A second variant of $\Pi_n(\yy)$ is the {\em $m$-generalized Pitman--Stanley polytope} for a positive integer $m$ defined as follows
\[
\Pi^{(m)}_n(\yy) := \left\{ (x_{i,j}) \in \mathbb{R}_{\geq 0}^{nm} : \enspace \sum_{i=1}^k x_{i,m} \leq \cdots \leq \sum_{i=1}^k x_{i,2} \leq \sum_{i=1}^k x_{i,1} \leq \sum_{i=1}^k y_i \text{ for each  $k=1,\ldots,n$}\right\}.
\]

The original $\Pi_n(\yy)$ corresponds to the case $m=1$. The lattice points of this polytope correspond to plane partitions of the straight shape $(x_1+\cdots+x_n,\ldots,x_1)$ with entries $1,2,\ldots,m+1$. There is also a {\em $m$-generalized skew Pitman--Stanley polytope} refining the two variants mentioned above. The number of faces, Ehrhart polynomial, and volumes of these polytopes have been recently studied in \cite{gen_PS1, gen_PS2}. The following example shows not all $m$-generalized Pitman--Stanley polytopes have magic positive Ehrhart polynomials.

\begin{example}
  For $n=m=2$ and $\yy = (1,1)$ the Ehrhart polynomial  is $\ehr_{\Pi^{(2)}_2({\bf 1})}(t)=\frac23t^4 + \frac{19}{6}t^3 + \frac{16}{3}t^2 + \frac{23}{6}t + 1$ which is not magic positive, since $\mathscr{M}(\ehr_{\Pi^{(2)}_2({\bf 1})}(t))=-\frac{1}{6} t^{2} - \frac{1}{6} t + 1$.  
\end{example}

\bibliographystyle{amsalpha}
\bibliography{bibliography}

@article {branden,
    AUTHOR = {Br{\"a}nd{\'e}n, Petter},
     TITLE = {On linear transformations preserving the {P}\'{o}lya frequency
              property},
   JOURNAL = {Trans. Amer. Math. Soc.},
  FJOURNAL = {Transactions of the American Mathematical Society},
    VOLUME = {358},
      YEAR = {2006},
    NUMBER = {8},
     PAGES = {3697--3716},
      ISSN = {0002-9947},
   MRCLASS = {05A15 (05A05 05A19 20F55 26C10)},
  MRNUMBER = {2218995},
MRREVIEWER = {Toshihiro Watanabe},
       DOI = {10.1090/S0002-9947-06-03856-6},
       URL = {https://doi.org/10.1090/S0002-9947-06-03856-6},
}

@book {stanley-ec1,
    AUTHOR = {Stanley, Richard P.},
     TITLE = {Enumerative combinatorics. {V}olume 1},
    SERIES = {Cambridge Studies in Advanced Mathematics},
    VOLUME = {49},
   EDITION = {Second},
 PUBLISHER = {Cambridge University Press, Cambridge},
      YEAR = {2012},
     PAGES = {xiv+626},
      ISBN = {978-1-107-60262-5},
   MRCLASS = {05-02 (05A15 06-02)},
  MRNUMBER = {2868112},
}

@book{CatalanBook,
 author = {Stanley, Richard P.},
 title = {Catalan numbers},
 isbn = {978-1-107-07509-2; 978-1-107-42774-7; 978-1-139-87149-5},
 year = {2015},
 publisher = {Cambridge: Cambridge University Press},
 language = {English},
 doi = {10.1017/CBO9781139871495},
 zbMATH = {6417735},
 Zbl = {1317.05010}
}

@article {stanley-hstar,
    AUTHOR = {Stanley, Richard P.},
     TITLE = {Decompositions of rational convex polytopes},
   JOURNAL = {Ann. Discrete Math.},
  FJOURNAL = {Annals of Discrete Mathematics},
    VOLUME = {6},
      YEAR = {1980},
     PAGES = {333--342},
   MRCLASS = {52A43},
  MRNUMBER = {593545},
MRREVIEWER = {P. McMullen},
}

@incollection {braun-unimodality,
    AUTHOR = {Braun, Benjamin},
     TITLE = {Unimodality problems in {E}hrhart theory},
 BOOKTITLE = {Recent trends in combinatorics},
    SERIES = {IMA Vol. Math. Appl.},
    VOLUME = {159},
     PAGES = {687--711},
 PUBLISHER = {Springer, [Cham]},
      YEAR = {2016},
      ISBN = {978-3-319-24296-5; 978-3-319-24298-9},
   MRCLASS = {52B20 (05A20)},
  MRNUMBER = {3526428},
MRREVIEWER = {Ruriko\ Yoshida},
       DOI = {10.1007/978-3-319-24298-9\{_}27}

@incollection {stanley-unimodality,
    AUTHOR = {Stanley, Richard P.},
     TITLE = {Log-concave and unimodal sequences in algebra, combinatorics,
              and geometry},
 BOOKTITLE = {Graph theory and its applications: {E}ast and {W}est ({J}inan,
              1986)},
    SERIES = {Ann. New York Acad. Sci.},
    VOLUME = {576},
     PAGES = {500--535},
 PUBLISHER = {New York Acad. Sci., New York},
      YEAR = {1989},
   MRCLASS = {05E15 (05E10 20C15 52B20)},
  MRNUMBER = {1110850},
MRREVIEWER = {L. Bruce Richmond},
       DOI = {10.1111/j.1749-6632.1989.tb16434.x},
       URL = {https://doi.org/10.1111/j.1749-6632.1989.tb16434.x},
}

@article{Baldoni_Vergne_2008,
  title={Kostant partitions functions and flow polytopes},
  author={Baldoni, Welleda and Vergne, Mich\`ele},
  journal={Transform. Groups},
  volume={13},
  number={3-4},
  pages={447--469},
  year={2008},
  publisher={Springer}
}

@article {ferroni,
    AUTHOR = {{Ferroni}, Luis},
     TITLE = {Matroids are not {E}hrhart positive},
   JOURNAL = {Adv. Math.},
  FJOURNAL = {Advances in Mathematics},
    VOLUME = {402},
      YEAR = {2022},
     PAGES = {Paper No. 108337, 27},
      ISSN = {0001-8708},
   MRCLASS = {52B40 (05B35 52B20)},
  MRNUMBER = {4396506},
MRREVIEWER = {Margaret M. Bayer},
       DOI = {10.1016/j.aim.2022.108337},
       URL = {https://doi.org/10.1016/j.aim.2022.108337},
}

@article {hibi-higashitani-yoshida-tsuchiya,
    AUTHOR = {Hibi, Takayuki and Higashitani, Akihiro and Tsuchiya, Akiyoshi
              and Yoshida, Koutarou},
     TITLE = {Ehrhart polynomials with negative coefficients},
   JOURNAL = {Graphs Combin.},
  FJOURNAL = {Graphs and Combinatorics},
    VOLUME = {35},
      YEAR = {2019},
    NUMBER = {1},
     PAGES = {363--371},
      ISSN = {0911-0119},
   MRCLASS = {52B20 (52B11)},
  MRNUMBER = {3898396},
MRREVIEWER = {Tam\'{a}s L\'{a}szl\'{o}},
       DOI = {10.1007/s00373-018-1990-9},
       URL = {https://doi.org/10.1007/s00373-018-1990-9},
}

@ARTICLE{ferroni-higashitani,
       author = {{Ferroni}, Luis and {Higashitani}, Akihiro},
        title = "{Examples and counterexamples in Ehrhart theory}",
    JOURNAL = {EMS Surv. Math. Sci.},
  FJOURNAL = {EMS Surveys in Mathematical Sciences},
    year = {2024},
    note = {to appear.}
}

@incollection {liu,
    AUTHOR = {Liu, Fu},
     TITLE = {On positivity of {E}hrhart polynomials},
 BOOKTITLE = {Recent trends in algebraic combinatorics},
    SERIES = {Assoc. Women Math. Ser.},
    VOLUME = {16},
     PAGES = {189--237},
 PUBLISHER = {Springer, Cham},
      YEAR = {2019},
   MRCLASS = {05A15 (05A20 52B20)},
  MRNUMBER = {3969575},
MRREVIEWER = {Matthias Beck},
       DOI = {10.1007/978-3-030-05141-9\_6},
       URL = {https://doi.org/10.1007/978-3-030-05141-9_6},
}

@article {mcmullen,
    AUTHOR = {McMullen, Peter},
     TITLE = {Valuations and {E}uler-type relations on certain classes of
              convex polytopes},
   JOURNAL = {Proc. London Math. Soc. (3)},
  FJOURNAL = {Proceedings of the London Mathematical Society. Third Series},
    VOLUME = {35},
      YEAR = {1977},
    NUMBER = {1},
     PAGES = {113--135},
      ISSN = {0024-6115,1460-244X},
   MRCLASS = {52A25},
  MRNUMBER = {448239},
MRREVIEWER = {P.\ R.\ Goodey},
       DOI = {10.1112/plms/s3-35.1.113},
       URL = {https://doi.org/10.1112/plms/s3-35.1.113},
}

@article {ehrhart,
    AUTHOR = {Ehrhart, Eug\`{e}ne},
     TITLE = {Sur les poly\`{e}dres rationnels homoth\'{e}tiques \`{a} {$n$}
              dimensions},
   JOURNAL = {C. R. Acad. Sci. Paris},
  FJOURNAL = {Comptes Rendus Hebdomadaires des S\'{e}ances de l'Acad\'{e}mie
              des Sciences},
    VOLUME = {254},
      YEAR = {1962},
     PAGES = {616--618},
      ISSN = {0001-4036},
   MRCLASS = {10.25 (52.10)},
  MRNUMBER = {130860},
}

@article {lam-postnikov,
    AUTHOR = {Lam, Thomas and Postnikov, Alexander},
     TITLE = {Polypositroids},
   JOURNAL = {Forum Math. Sigma},
  FJOURNAL = {Forum of Mathematics. Sigma},
    VOLUME = {12},
      YEAR = {2024},
     PAGES = {Paper No. e42, 67},
      ISSN = {2050-5094},
   MRCLASS = {52B40 (05B35 20F55 52B12)},
  MRNUMBER = {4718184},
MRREVIEWER = {Winfried\ Hochst\"attler},
       DOI = {10.1017/fms.2024.11},
       URL = {https://doi.org/10.1017/fms.2024.11},
}

@article{avila-ferroni-morales-abstract,
    author = {Nicolas Avila and Luis Ferroni and Alejandro H. Morales},
    title = {Luck and magic for {P}itman--{S}tanley polytopes},
    year = {2026},
    note = {Extended abstract for FPSAC 2026}
}

@ARTICLE{branden-ferroni-jochemko,
       author = {{Br{\"a}nd{\'e}n}, Petter and {Ferroni}, Luis and {Jochemko}, Katharina},
        title = "{Preservation of inequalities under Hadamard products}",
      journal = {Trans. Amer. Math. Soc.},
         year = 2026,
       note = {to appear.}
}

@article {liu-xiao,
    AUTHOR = {Liu, Yanxin and Xiao, Qiqi},
     TITLE = {Preservation of log-concavity on linear transformations},
   JOURNAL = {Discrete Math.},
  FJOURNAL = {Discrete Mathematics},
    VOLUME = {349},
      YEAR = {2026},
    NUMBER = {7},
     PAGES = {Paper No. 115083},
      ISSN = {0012-365X,1872-681X},
   MRCLASS = {05A20 (15)},
  MRNUMBER = {5040015},
       DOI = {10.1016/j.disc.2026.115083},
       URL = {https://doi.org/10.1016/j.disc.2026.115083},
}

@misc{athanasiadis-xiao-yan,
         author = {{Athanasiadis}, Christos A. and {Xiao}, Qiqi and {Yan}, Xue},
        title = "{Lattice point enumeration of some arbor polytopes}",
      journal = {arXiv e-prints},
         year = 2026,
        month = mar,
          eid = {arXiv:2603.11654},
        pages = {arXiv:2603.11654},
          doi = {10.48550/arXiv.2603.11654},
archivePrefix = {arXiv},
       eprint = {2603.11654},
 primaryClass = {math.CO},
       adsurl = {https://ui.adsabs.harvard.edu/abs/2026arXiv260311654A}
}

@article{avila-ferroni-morales-hstar,
    author = {Nicolas Avila and Luis Ferroni and Alejandro H. Morales},
    title = {The $h^*$-polynomials of $\mathbf{y}$-generalized permutohedra},
    year = {2026},
    note = {to appear}
}

@misc{DHS,
       author = {{Deligeorgaki}, Danai and {Han}, Bin and {Solus}, Liam},
        title = "{Colored Multiset Eulerian Polynomials}",
      journal = {arXiv e-prints},
         year = 2024,
        month = jul,
          eid = {arXiv:2407.12076},
        pages = {arXiv:2407.12076},
          doi = {10.48550/arXiv.2407.12076},
archivePrefix = {arXiv},
       eprint = {2407.12076},
 primaryClass = {math.CO},
       adsurl = {https://ui.adsabs.harvard.edu/abs/2024arXiv240712076D}
}

@incollection {Yan_pf_handbook,
    AUTHOR = {Yan, Catherine H.},
     TITLE = {Parking functions},
 BOOKTITLE = {Handbook of enumerative combinatorics},
    SERIES = {Discrete Math. Appl. (Boca Raton)},
     PAGES = {835--893},
 PUBLISHER = {CRC Press, Boca Raton, FL},
      YEAR = {2015}
}

@article {ferroni-jochemko-schroter,
    AUTHOR = {Ferroni, Luis and Jochemko, Katharina and Schr\"{o}ter, Benjamin},
     TITLE = {Ehrhart polynomials of rank two matroids},
   JOURNAL = {Adv. in Appl. Math.},
  FJOURNAL = {Advances in Applied Mathematics},
    VOLUME = {141},
      YEAR = {2022},
     PAGES = {Paper No. 102410, 26},
      ISSN = {0196-8858},
   MRCLASS = {52B40 (05A15 05B35 26C10 52B20)},
  MRNUMBER = {4461609},
MRREVIEWER = {Ping Zhan},
       DOI = {10.1016/j.aam.2022.102410},
       URL = {https://doi.org/10.1016/j.aam.2022.102410},
}

@article {postnikov,
    AUTHOR = {Postnikov, Alexander},
     TITLE = {Permutohedra, associahedra, and beyond},
   JOURNAL = {Int. Math. Res. Not. IMRN},
  FJOURNAL = {International Mathematics Research Notices. IMRN},
      YEAR = {2009},
    NUMBER = {6},
     PAGES = {1026--1106},
      ISSN = {1073-7928,1687-0247},
   MRCLASS = {05E30},
  MRNUMBER = {2487491},
       DOI = {10.1093/imrn/rnn153},
       URL = {https://doi.org/10.1093/imrn/rnn153},
}

@article{kenyonparkingfunctions,
    AUTHOR = {Kenyon, Richard and Yin, Mei},
     TITLE = {Parking functions: from combinatorics to probability},
   JOURNAL = {Methodol. Comput. Appl. Probab.},
  FJOURNAL = {Methodology and Computing in Applied Probability},
    VOLUME = {25},
      YEAR = {2023},
    NUMBER = {1},
     PAGES = {Paper No. 32, 30},
      ISSN = {1387-5841,1573-7713},
   MRCLASS = {60C05 (05A16 05A19 60F10)},
  MRNUMBER = {4549917},
       DOI = {10.1007/s11009-023-10022-5},
       URL = {https://doi.org/10.1007/s11009-023-10022-5},
}

@ARTICLE{konoike,
       author = {{Konoike}, Masato},
        title = "{A new class of magic positive Ehrhart polynomials of reflexive polytopes}",
      journal = {arXiv e-prints},
         year = 2024,
        month = sep,
          eid = {arXiv:2409.16648},
        pages = {arXiv:2409.16648},
          doi = {10.48550/arXiv.2409.16648},
archivePrefix = {arXiv},
       eprint = {2409.16648},
 primaryClass = {math.CO},
       adsurl = {https://ui.adsabs.harvard.edu/abs/2024arXiv240916648K}
}

@ARTICLE{konoike2,
       author = {{Konoike}, Masato},
        title = "{On the magic positivity of Ehrhart polynomials of dilated polytopes}",
      journal = {arXiv e-prints},
         year = 2025,
        month = apr,
          eid = {arXiv:2504.21395},
        pages = {arXiv:2504.21395},
          doi = {10.48550/arXiv.2504.21395},
archivePrefix = {arXiv},
       eprint = {2504.21395},
 primaryClass = {math.CO},
       adsurl = {https://ui.adsabs.harvard.edu/abs/2025arXiv250421395K}
}

@article{gen_PS1,
    AUTHOR = {Dugan, William T. and Hegarty, Maura and Morales, Alejandro H.
              and Raymond, Annie},
     TITLE = {Generalized {P}itman--{S}tanley {P}olytope: {V}ertices and
              {F}aces},
   JOURNAL = {Discrete Comput. Geom.},
  FJOURNAL = {Discrete \& Computational Geometry. An International Journal
              of Mathematics and Computer Science},
    VOLUME = {74},
      YEAR = {2025},
    NUMBER = {2},
     PAGES = {492--543},
      ISSN = {0179-5376,1432-0444},
   MRCLASS = {05C21 (05A15 05A19 06A07 52B05)},
  MRNUMBER = {4961368},
       DOI = {10.1007/s00454-024-00704-3},
       URL = {https://doi.org/10.1007/s00454-024-00704-3},
}

@article{Pitman_Stanley_1999, 
 AUTHOR = {Stanley, Richard P. and Pitman, Jim},
     TITLE = {A polytope related to empirical distributions, plane trees,
              parking functions, and the associahedron},
   JOURNAL = {Discrete Comput. Geom.},
  FJOURNAL = {Discrete \& Computational Geometry. An International Journal
              of Mathematics and Computer Science},
    VOLUME = {27},
      YEAR = {2002},
    NUMBER = {4},
     PAGES = {603--634},
       DOI = {10.1007/s00454-002-2776-6},
       url={http://arxiv.org/abs/math/9908029}
}

@misc{gen_PS2,
title={Generalized {P}itman--{S}tanley polytopes: volume and {E}hrhart polynomials},
author={William T. Dugan and Maura Hegarty and Alejandro H. Morales and Annie Raymond},
note={to appear}}

@article{genPFPoly,
 author = {Hanada, Mitsuki and Lentfer, John and Vindas-Mel{\'e}ndez, Andr{\'e}s R.},
 title = {Generalized parking function polytopes},
 fjournal = {Annals of Combinatorics},
 journal = {Ann. Comb.},
 issn = {0218-0006},
 volume = {28},
 number = {2},
 pages = {575--613},
 year = {2024},
 doi = {10.1007/s00026-023-00671-1},
 zbMATH = {7851997},
 Zbl = {1539.05006}
}

@article {KonvPak,
    AUTHOR = {Konvalinka, Matja\v{z} and Pak, Igor},
     TITLE = {Triangulations of {C}ayley and {T}utte polytopes},
   JOURNAL = {Adv. Math.},
  FJOURNAL = {Advances in Mathematics},
    VOLUME = {245},
      YEAR = {2013},
     PAGES = {1--33},
       DOI = {10.1016/j.aim.2013.06.012}
}

@article {KonvPak2,
    AUTHOR = {Konvalinka, Matja\v{z} and Pak, Igor},
     TITLE = {Cayley compositions, partitions, polytopes, and geometric
              bijections},
   JOURNAL = {J. Combin. Theory Ser. A},
  FJOURNAL = {Journal of Combinatorial Theory. Series A},
    VOLUME = {123},
      YEAR = {2014},
     PAGES = {86--91},
       DOI = {10.1016/j.jcta.2013.11.008}
}

@article{StanleyYin,
    AUTHOR = {Stanley, Richard P. and Yin, Mei},
     TITLE = {Some enumerative properties of parking functions},
   JOURNAL = {Comb. Theory},
  FJOURNAL = {Combinatorial Theory},
    VOLUME = {5},
      YEAR = {2025},
    NUMBER = {4},
     PAGES = {Paper No. 13, 39},
      ISSN = {2766-1334},
   MRCLASS = {05A15 (05A19 60C05)},
  MRNUMBER = {5024855},
}

@article{harris2024parkingfunctionsfixedset,
       author = {{Harris}, Pamela E. and {Martinez}, Lucy},
        title = "{Parking functions with a fixed set of lucky cars}",
      journal = {arXiv e-prints},
         year = 2024,
        month = oct,
          eid = {arXiv:2410.08057},
        pages = {arXiv:2410.08057},
          doi = {10.48550/arXiv.2410.08057},
archivePrefix = {arXiv},
       eprint = {2410.08057},
 primaryClass = {math.CO},
       adsurl = {https://ui.adsabs.harvard.edu/abs/2024arXiv241008057H}
}

@ARTICLE{ferreri2025enumeratingvectorparkingfunctions,
       author = {{Ferreri}, Melanie and {Harris}, Pamela E. and {Martinez}, Lucy and {Swartz}, Eric},
        title = "{Enumerating Vector Parking Functions and their Outcomes Based on Specified Lucky Cars}",
      journal = {arXiv e-prints},
         year = 2025,
        month = aug,
          eid = {arXiv:2508.13917},
        pages = {arXiv:2508.13917},
          doi = {10.48550/arXiv.2508.13917},
archivePrefix = {arXiv},
       eprint = {2508.13917},
 primaryClass = {math.CO},
       adsurl = {https://ui.adsabs.harvard.edu/abs/2025arXiv250813917F}
}

@article{GesselSeo,
 author = {Gessel, Ira M. and Seo, Seunghyun},
 title = {A refinement of {Cayley}'s formula for trees},
 fjournal = {The Electronic Journal of Combinatorics},
 journal = {Electron. J. Comb.},
 issn = {1077-8926},
 volume = {11},
 number = {2},
 pages = {research paper r27, 23},
 year = {2006},
 url = {https://eudml.org/doc/125512},
 zbMATH = {5005208},
 Zbl = {1080.05005}
}

@ARTICLE{ferroni-morales-panova,
       author = {{Ferroni}, Luis and {Morales}, Alejandro H. and {Panova}, Greta},
        title = "{Skew shapes, Ehrhart positivity and beyond}",
      journal = {arXiv e-prints},
         year = 2025,
        month = mar,
          eid = {arXiv:2503.16403},
        pages = {arXiv:2503.16403},
          doi = {10.48550/arXiv.2503.16403},
archivePrefix = {arXiv},
       eprint = {2503.16403},
 primaryClass = {math.CO},
       adsurl = {https://ui.adsabs.harvard.edu/abs/2025arXiv250316403F}
}

@article{KungYan,
 author = {Kung, Joseph P. S. and Yan, Catherine},
 title = {Gon{\v{c}}arov polynomials and parking functions},
 fjournal = {Journal of Combinatorial Theory. Series A},
 journal = {J. Comb. Theory, Ser. A},
 issn = {0097-3165},
 volume = {102},
 number = {1},
 pages = {16--37},
 year = {2003},
 doi = {10.1016/S0097-3165(03)00009-8}
}
\end{document}